\documentclass[12pt]{article}%
\usepackage{amssymb}
\usepackage{eurosym}
\usepackage{amsfonts}
\usepackage{amsmath}
\usepackage{graphicx}
\usepackage{epstopdf}%
\providecommand{\U}[1]{\protect\rule{.1in}{.1in}}
\newtheorem{theorem}{Theorem}

\newtheorem{exercise}[theorem]{Examples}
\newtheorem{lemma}[theorem]{Lemma}
\newtheorem{notation}[theorem]{Notation}

\newtheorem{proposition}[theorem]{Proposition}
\newtheorem{remark}[theorem]{Remark}

\newtheorem{summary}[theorem]{Summary}
\newenvironment{proof}[1][Proof]{\noindent\textbf{#1.} }{\ \rule{0.5em}{0.5em}}
\ifx\pdfoutput\relax\let\pdfoutput=\undefined\fi
\newcount\msipdfoutput
\ifx\pdfoutput\undefined\else
\ifcase\pdfoutput\else
\ifx\paperwidth\undefined\else
\ifdim\paperheight=0pt\relax\else\pdfpageheight\paperheight\fi
\ifdim\paperwidth=0pt\relax\else\pdfpagewidth\paperwidth\fi
\fi\fi\fi
\begin{document}

\title{Topology of Stokes Complex Related to a Polynomial Quadratic Differential :
Phase Transitions and Number of Short Trajectories}
\author{Gliia Braek, Mondher Chouikhi, and Faouzi Thabet}
\maketitle

\begin{abstract}
In this paper, we introduce ($\mathcal{D\diagup SG}$) correspondence to give a
full classification to the critical graphs (topology of Stokes complex) of
polynomial quadratic differentials $A\left(  z-a\right)  \left(
z^{2}-1\right)  dz^{2}$ on the Riemann sphere $\widehat{%
\mathbb{C}
}$, where $\left(  A,a\right)  \in%
\mathbb{C}
^{\ast}\times%
\mathbb{C}
$. We prove that the number of short trajectories depends on the location of
$a$ in $\Xi_{\theta}$, identified as union of certain curves defined in the
complex plane as the level sets of some harmonic functions. We point in the
"\textit{phase transition" }on\textit{\ }the topology of $\Xi_{\theta}$ which
leads to a double \textit{"tree case" }for some value of $\arg A$.

\end{abstract}

\textit{2020 Mathematics subject classification: 34C05, 57R05, 57R30}

\textit{Keywords and phrases: }Quadratic differentials. Horizontal and
vertical trajectories. Stokes complex. Half-plane domain. Strip domain.
Teichm\"{u}ller Lemma.

\section{Introduction}

This work is the first step of a project highlighting the accordance between
properties of solutions to a second order linear ordinary differential
equation (ODE) with polynomial coefficients in complex domain, and the
structure of related \textit{"parametric"} quadratic differential (or Stokes
geometry). A global asymptotic study of these ODE's can be found in Sibuya's
book (\cite{sibuya}).

Quadratic differentials on a Riemann surface $\mathcal{R}$ equipped with a
conformal structure $\left\{  U_{\alpha},\varphi_{\alpha}\right\}  $ is given
locally by an expression $\varpi(z)dz^{2},$ with $\varpi$ is holomorphic (or
meromorphic) function. These differentials define two important tools on
$\mathcal{R}$: a flat metric with singularities at the critical points of
$\varpi$ (zeros and poles), and trajectories (or foliation). Trajectories play
an important role in different area of mathematics and mathematical physics.
In particular, various geodesics with respect to the flat metric are called
\textit{short trajectories }(for definition, see \ref{para1}). For the general
theory of quadratic differentials, we refer to

(\cite{strebel},\cite{jenkins},\cite{jenk-spencer},\cite{bridgeland+smith},...).

By the word \textit{parametric, }we mean a family of such differentials that
depend on extra data. Let $\mathcal{D}$ be the set of these data. As these
parameters or data varies, one may ask how the structure of trajectories (or
Stokes geometry) change. The study of this correspondence between
$\mathcal{D}$ and Stokes geometry ($\mathcal{D\diagup SG}$) is a highly
attractive subject in mathematical physics. In quantum mechanics, for
instance, the Stokes geometry are relevant in the asymptotic study of
one-dimensional Schr\"{o}dinger equation (\cite[chapter3]{fedoryuk},
\cite{shapiro ev}, \cite{kawaii}, \cite{chouikhi+thabet}...). In this work, we
give an example of ($\mathcal{D\diagup SG}$) correspondence via harmonic
analysis. Before going any further, let give the main motivation to this work:

In \cite{maosero}, the distribution of poles of a particular solution $f$
(called \textit{int\'{e}grale tritronqu\'{e}e}) to the Painlev\'{e}-first
equation (P-I) has been studied :%
\[
\left\{
\begin{array}
[c]{c}%
f^{\prime\prime}(z)=6f^{2}(z)-z,\text{ }z\in%
\mathbb{C}%
\\
\\
f(z)\sim-\sqrt{\frac{z}{6}}\text{ \ \ \emph{if} }\left\vert \arg z\right\vert
<\frac{4\pi}{5}%
\end{array}
\right.
\]
We also mention \cite{Boutroux}, \cite{Ilpo laine}, \cite{maosero}. The main
result in \cite{maosero} is : $a$ \textit{is a pole of }$f$\textit{ if and
only if there exist }$b\in%
\mathbb{C}
$\textit{ such that the following Schr\"{o}dinger equation}%
\begin{equation}
\frac{d^{2}\psi(x)}{dx^{2}}=(4x^{3}-2ax-28b)\psi(x) \label{maosero-cubic}%
\end{equation}
\textit{admits two simultaneous different quantization conditions, where
}$\psi(x)$\textit{ is obtained by} \textit{isomonodromic deformation }(for
details, see also \cite{Kapeav}). A classification of Stokes graph related to
cubic oscillator has been made through proving that $(a,b)\in%
\mathbb{C}
^{2}$ in (\ref{maosero-cubic}) gives rise to a particular $\mathcal{SG}$
structure called \emph{Boutroux curve}\textit{ }(general ideas about this
topic can be found in \cite{bertola}). We reveal important questions from
Maosero's work : how the $\mathcal{SG}$ structure varies as $(a,b)$ varies in
$%
\mathbb{C}
^{2}$ (or a sub-manifold of $%
\mathbb{C}
^{2}$)?\label{first question} For a given pole $a$ of $f,$ does there exist a
unique $b$ such that (\ref{maosero-cubic}) has a solution?
\label{second question}In the present work, we answer these questions. More
precisely, consider the quadratic differential on the Riemann sphere
\begin{equation}
\varpi_{\theta,a}(z)dz^{2}=-\exp(i2\theta)\left(  z-a\right)  \left(
z^{2}-1\right)  dz^{2}, \label{our differential}%
\end{equation}
where $\theta\in\lbrack0,\pi\lbrack,a\in%
\mathbb{C}
.$ The data set
\[
\mathcal{D}=\left\{  \left(  \theta,a\right)  \in\lbrack0,\pi\lbrack\times%
\mathbb{C}
\text{ }\right\}
\]
is a complex sub-manifold of $%
\mathbb{C}
^{2}.$ We aim at studying all possible $\mathcal{SG}$ structures and
topological changes related to trajectories of $\varpi_{\theta,a}$ as $\left(
\theta,a\right)  $ varies in $\mathcal{D}.$ For $\theta\in\lbrack0,\pi
\lbrack,$ we consider the set:%
\begin{equation}
\Xi_{\theta}=\left\{  a\in%
\mathbb{C}
\setminus\left\{  \pm1\right\}  :\varpi_{\theta,a}(z)dz^{2}\text{ admits at
least a short trajectory}\right\}  \label{our set imp}%
\end{equation}
The ($\mathcal{D\diagup SG}$) correspondence is expressed via a surjective map
(see summary.\ref{summary})
\[
\mathcal{J}:\mathcal{D\longrightarrow SG}.
\]
As $\theta$ varies in $[0,\pi\lbrack,$ the occurrence of \textit{tree cases
}(or Boutroux curve from Maosero's work \cite{maosero}) changes from $0$ to
$2.$ We call this phenomenon \textit{phase transitions }in the
($\mathcal{D\diagup SG}$) correspondence.

In \cite[problem 3]{shapiro ev}, the author asked if it is true that for any
polynomials $P,$ there exist $t\in\left[  0,2\pi\right[  $ such that the
quadratic differential $\exp(it)P(z)dz^{2}$ is \emph{very flat} \cite[Problem
3]{shapiro ev}. A positive answer is presented in the case $\deg P=3;$ (see
summary.\ref{summary}).

The paper is organized as follow: In the first section, we introduce
$\Sigma_{-1,\theta},\Sigma_{1,\theta},$ and $\Sigma_{\theta}$ as a union of
level sets of harmonic functions induced by the differential $\varpi
_{\theta,a}$. We give a global description of the topology and connected
components of $\Sigma_{\pm1,\theta}\cup\Sigma_{\theta}$. As $\theta$ varies
from $0$ to $\pi,$ the number of self intersections points changes. These
remarkable transition phases induce a mutation on the Stokes complex of
$\varpi_{\theta,a}.$ More focus will be on the cases $\theta=0$ and
$\theta=\frac{\pi}{4}.$ The main result of the second part of the paper
(theorem.\ref{main thm}) resulting in the description of the critical graph of
$\varpi_{\theta,a}$ and the definition of the map $\mathcal{J}$ $.$ For a
fixed $\theta$ in $[0,\pi\lbrack,$ we prove that the set $\Xi_{\theta}$ is the
union of some analytic curves deduced from $\Sigma_{\pm1,\theta}\cup
\Sigma_{\theta}.$This fact helps us proving the surjectivity of $\mathcal{J}$
in summary.\ref{summary}.

\section{Level set of harmonic functions\label{para1}}

For $\theta\in\left[  0,\pi/2\right[  ,$ we consider the sets%
\begin{align*}
\Sigma_{1,\theta}  &  =\left\{  a\in%
\mathbb{C}
\setminus\left]  -\infty,-1\right]  :\Re\left(  \int_{\left[  1,a\right]
}e^{i\theta}\sqrt{p_{a}\left(  z\right)  }dz\right)  =0\right\}  ;\\
\Sigma_{-1,\theta}  &  =\left\{  a\in%
\mathbb{C}
\setminus\left[  1,+\infty\right[  :\Re\left(  \int_{\left[  -1,a\right]
}e^{i\theta}\sqrt{p_{a}\left(  z\right)  }dz\right)  =0\right\}  ;\\
\Sigma_{\theta}  &  =\left\{  a\in%
\mathbb{C}
\setminus\left[  -1,1\right]  :\Re\left(  \int_{\left[  -1,1\right]
}e^{i\theta}\sqrt{p_{a}\left(  z\right)  }dz\right)  =0\right\}  ;
\end{align*}
where $p_{a}\left(  z\right)  $ is the complex polynomial defined by
\[
p_{a}\left(  z\right)  =\left(  z-a\right)  \left(  z^{2}-1\right)  .
\]

\begin{lemma}
\label{level1}Let $\theta\in\left[  0,\pi/2\right[  .$ Then,

each of the sets $\Sigma_{1,\theta}$ and $\Sigma_{-1,\theta}$ is formed by two
smooth curves that are locally orthogonal respectively at $z=1$ and $z=-1;$
more precisely :%
\begin{align*}
\lim\limits_{\substack{a\longrightarrow-1\\a\in\Sigma_{-1,\theta}}}\arg\left(
a+1\right)   &  =\frac{-2\theta+\left(  2k+1\right)  \pi}{4},k=0,1,2,3;\\
\lim\limits_{\substack{a\longrightarrow+1\\a\in\Sigma_{1,\theta}}}\arg\left(
a-1\right)   &  =\frac{-\theta+k\pi}{2},k=0,1,2,3.
\end{align*}
Curves defining $\Sigma_{1,\theta}$ (resp. $\Sigma_{-1,\theta}$ ) meet only in
$z=1$ (resp. $z=-1$ ). Moreover, for $\theta\notin\left\{  0,\frac{\pi}%
{2}\right\}  ,$ they diverge differently to $\infty$ following one of the
directions
\[
\lim_{\substack{\left\vert a\right\vert \longrightarrow+\infty\\a\in
\Sigma_{\pm1,\theta}}}\arg a=\frac{-2\theta+2k\pi}{5},k=0,1,2,3,4.
\]
For $\theta=0,$ (resp. $\theta=\frac{\pi}{2}$ ), one ray of $\Sigma_{1,\theta
}$ (resp. $\Sigma_{-1,\theta}$ ) diverges to $z=-1$ (resp. $z=1$)
\end{lemma}

\begin{proof}
Given a non-constant harmonic function $u$ defined in some domain
$\mathcal{D}$ of $%
\mathbb{C}
,$ the critical points of $u$ are precisely those where%
\[
\frac{\partial u}{\partial z}=\frac{1}{2}\left(  \frac{\partial u}{\partial
x}-i\frac{\partial u}{\partial y}\right)  =0.
\]
They are isolated. If $v$ is the harmonic conjugate of $u$ in $\mathcal{D},$
say, $f\left(  z\right)  =u\left(  z\right)  +iv\left(  z\right)  $ is
analytic in $\mathcal{D},$ then, by Cauchy-Riemann,
\[
f^{\prime}\left(  z\right)  =0\Longleftrightarrow u^{\prime}\left(  z\right)
=0.
\]
The level-set
\[
\Sigma_{z_{0}}=\left\{  z\in\mathcal{D}:u\left(  z\right)  =u\left(
z_{0}\right)  \right\}
\]
of $u$ through a point $z_{0}\in\mathcal{D}$ depends on the behavior of $f$
near $z_{0}.$ More precisely, if $z_{0}$ is a critical point of $u,$ (
$u^{\prime}\left(  z_{0}\right)  =0$ ), then, there exist a neighborhood
$\mathcal{U}$ of $z_{0},$ a holomorphic function $g\left(  z\right)  $ defined
in $\mathcal{U},$ such that
\[
\forall z\in\mathcal{U},f\left(  z\right)  =\left(  z-z_{0}\right)
^{m}g\left(  z\right)  ;g\left(  z\right)  \neq0.
\]
Taking a branch of the $m$-th root of $g\left(  z\right)  ,$ $f$ has the local
structure%
\[
f\left(  z\right)  =\left(  h\left(  z\right)  \right)  ^{m},\forall
z\in\mathcal{U}.
\]
It follows that $\Sigma_{z_{0}}$ is locally formed by $m$ analytic arcs
passing through $z_{0}$ and intersecting there with equal angles $\pi/m.$
Through a regular point $z_{0}\in\mathcal{D},$ ( $u^{\prime}\left(
z_{0}\right)  \neq0$ ), the implicit function theorem asserts that
$\Sigma_{z_{0}}$ is locally a single analytic arc. Notice that the level-set
of a harmonic function cannot terminate at a regular point; (see
\cite{harmonic1},\cite{harmonic2}...).

Let consider the multi-valued function
\[
f_{1,\theta}\left(  a\right)  =\int_{1}^{a}e^{i\theta}\sqrt{p_{a}\left(
t\right)  }dt,a\in%
\mathbb{C}
.
\]
Integrating along the segment $\left[  1,a\right]  ,$ we may assume without
loss of generality that
\begin{equation}
f_{1,\theta}\left(  a\right)  =ie^{i\theta}\left(  a-1\right)  ^{2}\int
_{0}^{1}\sqrt{t\left(  1-t\right)  }\sqrt{t\left(  a-1\right)  +2}dt=\left(
a-1\right)  ^{2}g\left(  a\right)  ;g\left(  1\right)  \neq0.
\label{integ defining g}%
\end{equation}
It is obvious that :
\[
\forall a\in%
\mathbb{C}
\setminus\left]  -\infty,-1\right]  ,\left\{  t\left(  a-1\right)
+2;t\in\left[  0,1\right]  \right\}  =\left[  2,a+1\right]  \subset%
\mathbb{C}
\setminus\left]  -\infty,0\right]  .
\]
Therefore, with a fixed choice of the argument and square root inside the
integral, $f_{1,\theta}$ and $g$ are single-valued analytic functions in $%
\mathbb{C}
\setminus\left]  -\infty,-1\right]  .$

Suppose that for some $a\in%
\mathbb{C}
\setminus\left]  -\infty,-1\right]  ,a\neq1,$%
\[
u\left(  a\right)  =\Re f_{1,\theta}\left(  a\right)  =0;f_{1,\theta}^{\prime
}\left(  a\right)  =0.
\]
Then,
\[
\left(  a-1\right)  ^{3}g^{\prime}\left(  a\right)  +2f_{1,\theta}\left(
a\right)  =0.
\]
Taking the real parts, we get
\begin{align*}
0  &  =\int_{0}^{1}\sqrt{t\left(  1-t\right)  }\Im\left(  e^{i\theta}\left(
a-1\right)  ^{2}\sqrt{t\left(  a-1\right)  +2}\right)  dt;\\
0  &  =\Re\left(  \left(  a-1\right)  ^{3}g^{\prime}\left(  a\right)  \right)
=\int_{0}^{1}t\sqrt{t\left(  1-t\right)  }\Im\left(  \frac{e^{i\theta}\left(
a-1\right)  ^{3}}{2\sqrt{t\left(  a-1\right)  +2}}\right)  dt.
\end{align*}

By continuity of the functions inside these integrals along the segment
$\left[  0,1\right]  ,$ there exist $t_{1},t_{2}\in\left[  0,1\right]  $ such
that
\[
\Im\left(  e^{i\theta}\left(  a-1\right)  ^{2}\sqrt{t_{1}\left(  a-1\right)
+2}\right)  =\Im\left(  \frac{e^{i\theta}\left(  a-1\right)  ^{3}}%
{2\sqrt{t_{2}\left(  a-1\right)  +2}}\right)  =0;
\]
and then
\[
e^{2i\theta}\left(  a-1\right)  ^{4}\left(  t_{1}\left(  a-1\right)
+2\right)  >0,\left(  \frac{e^{2i\theta}\left(  a-1\right)  ^{6}}{t_{2}\left(
a-1\right)  +2}\right)  >0.
\]
Taking their ratio, we get
\[
\frac{\left(  t_{1}\left(  a-1\right)  +2\right)  \left(  t_{2}\left(
a-1\right)  +2\right)  }{\left(  a-1\right)  ^{2}}>0.
\]
which cannot hold, since, if $\Im a>0,$ then%
\[
0<\arg\left(  t_{1}\left(  a-1\right)  +2\right)  +\arg\left(  \left(
t_{2}\left(  a-1\right)  +2\right)  \right)  <2\arg\left(  a+1\right)
<\arg\left(  \left(  a-1\right)  ^{2}\right)  <2\pi.
\]
The case $\Im a<0$ is in the same vein, while the case $a\in%
\mathbb{R}
$ can be easily discarded. Thus, $a=1$ is the unique critical point of $\Re
f_{1,\theta}.$ Since $f^{\prime}{}^{\prime}{}_{1,\theta}(1)=2g\left(
1\right)  \neq0,$ we deduce the local behavior of $\Sigma_{1,\theta}$ near
$a=1.$

Suppose now that for some $\theta\in\left]  0,\frac{\pi}{2}\right[  ,$ a ray
of $\Sigma_{\pm1,\theta}$ diverges to a certain point in $\left]
-\infty,-1\right[  ;$ for example,%
\[
\left(  \overline{\Sigma_{1,\theta}}\setminus\Sigma_{1,\theta}\right)
\cap\left\{  z\in\overline{%
\mathbb{C}
}:\Im z\geq0\right\}  =\left\{  x_{\theta}\right\}  .
\]
Let $\epsilon>0$ such that $0<\theta-2\epsilon.$ For $a\in\Sigma_{1,\theta}$
satisfying $\pi-\epsilon<\arg a<\pi,$ and using the Riemann sum of integrals,
we get
\[
0<\theta-2\epsilon<\theta+2\arg\left(  a-1\right)  +\arg\int_{0}^{1}%
\sqrt{t\left(  1-t\right)  }\sqrt{t\left(  a-1\right)  +2}dt<\frac{\pi}%
{2}+\theta-\frac{\epsilon}{2}<\pi,
\]
and then $\Re f_{1,\theta}\left(  a\right)  \neq0.$The other cases are
similar. Thus, any ray of $\Sigma_{\pm1,\theta}$ should diverge to $\infty.$
The case $\theta=0$ is quiet easier.

It is obvious that for $a\longrightarrow\infty,$ we have $\left\vert
f_{1,\theta}\left(  a\right)  \right\vert \longrightarrow+\infty.$ In
particular, if $a\longrightarrow\infty$ and $\Re f_{1,\theta}\left(  a\right)
=0$ then, $\left\vert \Im f_{1,\theta}\left(  a\right)  \right\vert
\longrightarrow+\infty.$ It follows that
\[
\arg\left(  f_{1,\theta}\left(  a\right)  \right)  \sim\arg\left(  \frac
{4}{15}e^{i\theta}a^{5/2}\right)  \longrightarrow\frac{\pi}{2}+k\pi,k\in%
\mathbb{Z}
\text{ as }a\longrightarrow\infty,\Re f_{1,\theta}\left(  a\right)  =0.
\]
We get the behavior of any arc of $\Sigma_{1,\theta}$ that diverges to
$\infty.$ In particular, from the maximum modulus principle, two rays of
$\Sigma_{1,\theta}$ cannot diverge to $\infty$ in the same direction.

If $\Sigma_{1,\theta}$ contains a regular point $z_{0}$ (for example, $\Im
z_{0}>0$) not belonging to the arcs of $\Sigma_{1,\theta}$ emerging from
$a=1.$ The two rays of the level set curve $\gamma$ going through $z_{0}$
diverge to $\infty$ in two different directions in the upper half-plane. It
follows that $\gamma$ should go through $z_{1}=1+iy,$ for some $y>0,$ or
$z_{1}=y,$ for some $y>1.$ It is easy to check that in the two cases, for any
choice of the argument,
\[
\Re\int_{1}^{z_{1}}\left(  e^{i\theta}\sqrt{p_{z_{1}}\left(  t\right)
}dt\right)  \neq0;
\]
and we get a contradiction. Thus, $\Sigma_{1,\theta}$ is formed only by the
two curves going through $a=1.$ The same idea gives the structure of
$\Sigma_{-1,\theta};$ even more, from the relation
\begin{equation}
\Re f_{\pm1,\theta}\left(  a\right)  =0\Longleftrightarrow\Re f_{\pm
1,\frac{\pi}{2}-\theta}\left(  -\overline{a}\right)  =0, \label{relat symm}%
\end{equation}
available for any $\theta\in\left[  \pi/4,\pi/2\right[  ,$ one can see easily
that $\Sigma_{-1,\frac{\pi}{2}-\theta}$ and $\Sigma_{1,\theta}$ are symmetric
with respect to the imaginary axis; (\ref{relat symm}). This, leads us to
restrict our investigation to the case $\theta\in\left[  0,\pi/4\right]  .$
\end{proof}

In the squel, we define the functions defined as in the proof of
lemma.\ref{level1} :
\[
f_{\theta}\left(  a\right)  =e^{i\theta}\int_{-1}^{1}\sqrt{p_{a}\left(
t\right)  }dt;f_{\pm1,\theta}\left(  a\right)  =e^{i\theta}\int_{\pm1}%
^{a}\sqrt{p_{a}\left(  t\right)  }dt.
\]

\begin{proposition}
\label{level2}For $\theta\in\left[  0,\pi/4\right]  ,$ we denote by
$\Sigma_{\pm1,\theta}^{l}$ and $\Sigma_{\pm1,\theta}^{r}$ respectively the
left and right curves defining $\Sigma_{\pm1,\theta}$ in the upper half plane,
see Fig.\ref{FIG2left}. Then%
\[%
\begin{tabular}
[c]{|l|l|l|l|l|}\hline
& $\Sigma_{1,\theta}^{l}\cap\Sigma_{-1,\theta}^{r}$ & $\Sigma_{1,\theta}%
^{l}\cap\Sigma_{-1,\theta}^{l}$ & $\Sigma_{1,\theta}^{r}\cap\Sigma_{-1,\theta
}^{r}$ & $\Sigma_{1,\theta}^{r}\cap\Sigma_{-1,\theta}^{l}$\\\hline
$\theta\in\left[  0,\pi/8\right[  $ & $\left\{  t_{\theta}\right\}  $ &
$\left\{  e_{\theta}\right\}  $ & $\emptyset$ & $\emptyset$\\\hline
$\theta\in\left[  \pi/8,\pi/4\right]  $ & $\left\{  t_{\theta}\right\}  $ &
$\emptyset$ & $\emptyset$ & $\emptyset$\\\hline
\end{tabular}
\ \ \ \ \ \ \ \ \ \ \
\]
where $t_{\theta}$ and $e_{\theta}$ are two complex number varying
respectively in two smooth curves
\[
\mathcal{A\cap}\left\{  a\in%
\mathbb{C}
:-1\leq\Re a\leq0,\Im a\geq0\right\}  ,
\]
and
\[
\mathcal{E\cap}\left\{  a\in%
\mathbb{C}
:\Re a\leq-1,\Im a\geq0\right\}  ;
\]
where $\mathcal{A}$ and $\mathcal{E}$ are defined by
\begin{align*}
\mathcal{A}  &  \mathcal{=}\left\{  a\in%
\mathbb{C}
:-1\leq\Re a\leq1,\Re f_{1,\theta}\left(  a\right)  =\Re f_{-1,\theta}\left(
a\right)  =0\right\}  ;\\
\mathcal{E}  &  \mathcal{=}\left\{  a\in%
\mathbb{C}
:\left\vert \Re a\right\vert \geq1,\Re f_{1,\theta}\left(  a\right)  =\Re
f_{-1,\theta}\left(  a\right)  =0\right\}  .
\end{align*}
Moreover, the set $\Sigma_{\theta}$ is a smooth curve included in the part of
the upper half plane of the strip bounded by the segment $\left[  -1,1\right]
$ and the lines $y=-\tan\left(  2\theta\right)  \left(  x\pm1\right)  .$ It
goes through $t_{\theta}$ and $e_{\theta}$ (if exists); its two rays diverge
differently to $\infty$ following the direction $\arg a=\pi-2\theta,$ and to
the unique $s_{\theta}\in\left[  -1,0\right]  ,$ such that
\[
\Re\int_{-1}^{1}e^{i\theta}\left(  \sqrt{p_{s_{\theta}}\left(  t\right)
}\right)  _{+}dt=0.
\]
In particular
\[
\Sigma_{1,\theta}\cap\Sigma_{-1,\theta}\cap\left\{  a\in%
\mathbb{C}
:\Im a<0\right\}  =\emptyset.
\]

\end{proposition}

\begin{figure}[th]
\centering\includegraphics[width=4in]{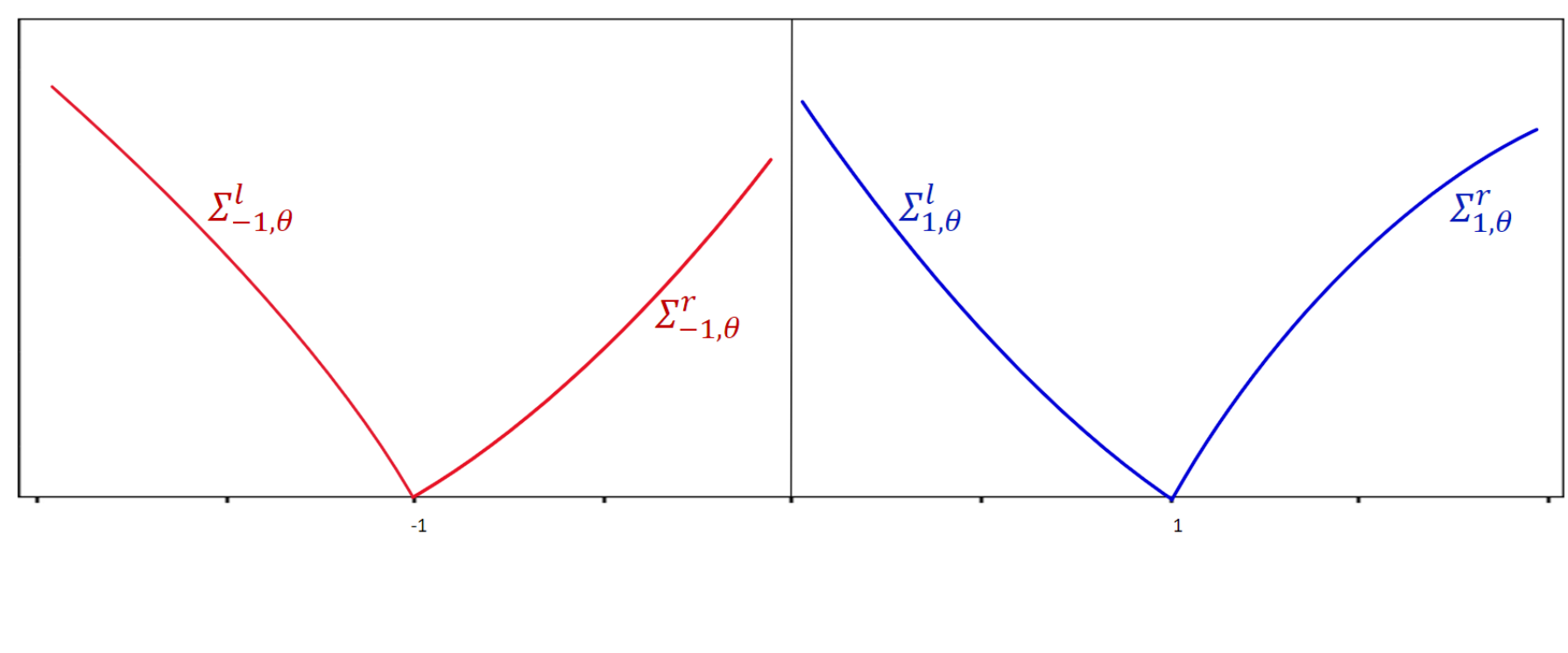} \caption{Sets $\Sigma
_{\pm1,\theta}^{l}$ and $\Sigma_{\pm1,\theta}^{r}$}%
\label{FIG2left}%
\end{figure}

\begin{proof}
If we denote by
\begin{align*}
\Sigma_{\pm1,\theta}^{l}\cap\left\{  z\in%
\mathbb{C}
:\Im z\geq0\right\}   &  =\left\{  a_{\pm1}^{l}\left(  s\right)  ,s\in\left[
0,+\infty\right[  ,a_{\pm1}^{l}\left(  0\right)  =\pm1\right\}  ;\\
\Sigma_{\pm1,\theta}^{r}\cap\left\{  z\in%
\mathbb{C}
:\Im z\geq0\right\}   &  =\left\{  a_{\pm1}^{r}\left(  s\right)  ,s\in\left[
0,+\infty\right[  ,a_{\pm1}^{r}\left(  0\right)  =\pm1\right\}  ,
\end{align*}
then, with (\ref{integ defining g}),%
\begin{align}
\Im\left(  e^{i\theta}\left(  a_{1}^{l}\left(  s\right)  -1\right)  ^{2}%
\int_{0}^{1}\sqrt{t\left(  1-t\right)  }b\left(  t\right)  dt\right)   &
=0;\label{eq1}\\
\Im\left(  e^{i\theta}\left(  a_{1}^{l}\right)  ^{\prime}\left(  s\right)
\left(  a_{1}^{l}\left(  s\right)  -1\right)  \int_{0}^{1}\sqrt{\frac{t}{1-t}%
}b\left(  t\right)  dt\right)   &  =0; \label{eq2}%
\end{align}
where $b\left(  t\right)  =\sqrt{t\left(  a_{1}^{l}\left(  s\right)
-1\right)  +2}.$ Taking the arguments in $\left[  -\pi,\pi\right[  ,$ we get
\begin{align}
\theta+2\arg\left(  a_{1}^{l}\left(  s\right)  -1\right)  +\arg\int_{0}%
^{1}\sqrt{t\left(  1-t\right)  }b\left(  t\right)  dt  &  \equiv
0\operatorname{mod}\left(  \pi\right)  ;\label{eq4}\\
\theta+\arg\left(  a_{1}^{l}\right)  ^{\prime}\left(  s\right)  +\arg\left(
a_{1}^{l}\left(  s\right)  -1\right)  +\arg\int_{0}^{1}\sqrt{\frac{t}{1-t}%
}b\left(  t\right)  dt  &  \equiv0\operatorname{mod}\left(  \pi\right)  ,
\label{eq5}%
\end{align}
and then
\[
\arg\frac{\left(  a_{1}^{l}\right)  ^{\prime}\left(  s\right)  }{a_{1}%
^{l}\left(  s\right)  -1}\equiv\arg\int_{0}^{1}\sqrt{t\left(  1-t\right)
}b\left(  t\right)  dt-\arg\int_{0}^{1}\sqrt{\frac{t}{1-t}}b\left(  t\right)
dt\operatorname{mod}\left(  \pi\right)  .
\]
Using the Riemann sum of integrals, and the obvious inequality
\[
\arg\left(  x+y\right)  <\arg\left(  \alpha x+\beta y\right)  ,
\]
available for $x,y\in%
\mathbb{C}
,0<\arg x<\arg y<\frac{\pi}{2},0<\alpha<\beta.$ We get for $s>0,$
\[
-\frac{\arg\left(  a\left(  s\right)  +1\right)  }{2}<\arg\int_{0}^{1}%
\sqrt{t\left(  1-t\right)  }b\left(  t\right)  dt-\arg\int_{0}^{1}\sqrt
{\frac{t}{1-t}}b\left(  t\right)  dt\leq0.
\]
Continuity of the function $s\longmapsto\arg\frac{\left(  a_{1}^{l}\right)
^{\prime}\left(  s\right)  }{a_{1}^{l}\left(  s\right)  -1}$ in $\left]
0,+\infty\right[  $ gives a certain $k\in\left\{  0,1\right\}  ,$ such that :%
\[
\forall s\geq0,-\frac{\pi}{2}+k\pi<-\frac{\arg\left(  a_{1}^{l}\left(
s\right)  +1\right)  }{2}+k\pi<\arg\frac{\left(  a_{1}^{l}\right)  ^{\prime
}\left(  s\right)  }{a_{1}^{l}\left(  s\right)  -1}\leq k\pi.
\]
In other word, $\Im\left(  \frac{\left(  a_{1}^{l}\right)  ^{\prime}\left(
s\right)  }{a_{1}^{l}\left(  s\right)  -1}\right)  $ and $\Re\left(
\frac{\left(  a_{1}^{l}\right)  ^{\prime}\left(  s\right)  }{a_{1}^{l}\left(
s\right)  -1}\right)  $ preserve the same signs in $\left[  0,+\infty\right[
,$ and we get the monotony of the functions $s\longmapsto\arg\left(  a_{1}%
^{l}\left(  s\right)  -1\right)  ,$ and $s\longmapsto\left\vert a_{1}%
^{l}\left(  s\right)  -1\right\vert $ in $\left[  0,+\infty\right[  .$ Taking
into account their boundary values at $s=0$ and $\infty,$ we get
\begin{equation}
\forall s\geq0,-\frac{\arg\left(  a_{1}^{l}\left(  s\right)  +1\right)  }%
{2}<\arg\left(  \frac{\left(  a_{1}^{l}\right)  ^{\prime}\left(  s\right)
}{a_{1}^{l}\left(  s\right)  -1}\right)  <0. \label{eq6}%
\end{equation}
In the same vein, one can show that
\begin{equation}
\forall t\geq0,\frac{\pi}{2}<\arg\left(  \frac{\left(  a_{-1}^{l}\right)
^{\prime}\left(  t\right)  }{a_{-1}^{l}\left(  t\right)  +1}\right)
<\frac{\pi}{2}+\frac{\arg\left(  a_{-1}^{l}\left(  t\right)  -1\right)  }{2}.
\label{eq7}%
\end{equation}
Since $\Re a_{-1}^{l}\left(  t\right)  <-1$ for any $t>0,$ it follows that, if
$a_{-1}^{l}\left(  t\right)  =a_{1}^{l}\left(  s\right)  ,$ then
\[
-\frac{\pi}{2}<\arg\left(  a_{-1}^{l}\left(  t\right)  +1\right)  -\arg\left(
a_{1}^{l}\left(  s\right)  -1\right)  <0;
\]
in the later case, after making the difference between (\ref{eq7}) and
(\ref{eq6}), we obtain
\[
0<\arg\left(  \frac{\left(  a_{-1}^{l}\right)  ^{\prime}\left(  t\right)
}{\left(  a_{1}^{l}\right)  ^{\prime}\left(  s\right)  }\right)  <\pi.
\]
Thus, $\Sigma_{1,\theta}^{l}$ and $\Sigma_{-1,\theta}^{l}$ cannot meet twice
in the upper half-plane. Repeating the same reasoning for the three other
cases, we get all the points of the lemma about the uniqueness of intersection
when exists.

For $\theta\in\left[  0,\pi/4\right]  ,$ there exist two directions that
should follow any curves of $\Sigma_{\pm1,\theta}$ diverging to $\infty$ in
the upper half plane, $D_{1}:\arg a=\frac{-2\theta+2\pi}{5}$ and $D_{2}:\arg
a=\frac{-2\theta+4\pi}{5}.$ More precisely, since for $a=\pm1,$ $\Sigma
_{a,\theta}^{r}$ and $\Sigma_{a,\theta}^{l}$ cannot diverge to $\infty$ in the
same direction, we deduce that, $\Sigma_{a,\theta}^{r}$ and $\Sigma_{a,\theta
}^{l}$ follow respectively the direction $D_{1}$ and $D_{2}.$ Therefore,
$\Sigma_{1,\theta}^{l}$ and $\Sigma_{-1,\theta}^{r}$ should meet at some
point, say $t_{\theta}.$ We claim that for $\theta\in\left[  0,\frac{\pi}%
{4}\right]  ,$%
\[
-1\leq\Re t_{\theta}\leq0.
\]
The first inequality is obvious since $\Sigma_{-1,\theta}^{r}\cap\left\{  z\in%
\mathbb{C}
:\Re z<-1\right\}  =\emptyset.$ If $\Re t_{\theta}>0,$ then,
\[
\left\vert \sqrt{t_{\theta}-t}\right\vert \leq\left\vert \sqrt{t_{\theta}%
+t}\right\vert ,\forall t\in\left[  0,1\right]  ,
\]
and
\[
0<\frac{\arg\left(  t_{\theta}+1\right)  }{2}\leq\arg\int_{-1}^{1}%
\sqrt{p_{t_{\theta}}\left(  t\right)  }=\arg\int_{0}^{1}\sqrt{1-t^{2}}\left(
\sqrt{t_{\theta}-t}+\sqrt{t_{\theta}+t}\right)  dt\leq\frac{\arg t_{\theta}%
}{2}<\frac{\pi}{4}.
\]
Therefore,
\begin{equation}
0<\arg\int_{-1}^{1}e^{i\theta}\sqrt{p_{t_{\theta}}\left(  t\right)  }%
<\frac{\pi}{2}, \label{resigmathetapositif}%
\end{equation}
and $t_{\theta}\notin\Sigma_{\theta};$ a contradiction. Moreover, relation
(\ref{relat symm}) shows that the sets $\Sigma_{-1,\pi/4}$ and $\Sigma
_{1,\pi/4}$ are symmetric with respect to the imaginary axis; it follows that
$a_{\frac{\pi}{4}}\subset i%
\mathbb{R}
^{+\ast}.$ The monotony of $\arg a_{1}^{l}\left(  s\right)  $ and $\arg
a_{-1}^{r}\left(  s\right)  $ prove that for $\theta\in\left[  0,\frac{\pi}%
{4}\right]  ,$
\begin{equation}
\Im t_{\theta}\in\left[  0,a_{\frac{\pi}{4}}/i\right]  . \label{a theta}%
\end{equation}
Implicit function theorem, and relation (\ref{a theta}), prove that the set
$\left\{  t_{\theta};\theta\in\left[  0,\pi\right]  \right\}  $ defined by
$\Re f_{1,\theta}\left(  t_{\theta}\right)  =\Re f_{1,\theta}\left(
t_{\theta}\right)  =0$ is a closed curve encircling $\left[  -1,1\right]  ,$
that is smooth in $\mathcal{A}\setminus\left\{  \pm1\right\}  .$ The relation
$t_{\theta}\in\Sigma_{1,\theta}^{l}\cap\Sigma_{-1,\theta}^{l}$ implies that%
\[
a_{0}=-1;\lim\limits_{\theta\longrightarrow0^{\pm}}\arg\left(  t_{\theta
}+1\right)  =\pm\frac{\pi}{4}.
\]
See Fig.\ref{FIG00}.

For $\Delta=\left[  -1,+\infty\right[  $ or $\Delta=\left]  -\infty,1\right]
,$ the function $\Re f_{\theta}$ is harmonic in $%
\mathbb{C}
\setminus\Delta,$ which implies the smoothness of any curve in its zero-level
set. If $\gamma$ is a ray of $\Sigma_{\theta}$ going through $t_{\theta}$ and
diverging to $\infty,$ then, from relation
\[
\left(  \int_{-1}^{1}e^{i\theta}\sqrt{p_{a}\left(  t\right)  }dt\right)
^{2}\sim\frac{\pi^{2}}{4}ae^{2i\theta},a\longrightarrow\infty,
\]
we deduce that%
\[
\arg a\sim\pi-2\theta,a\in\gamma,a\longrightarrow\infty.
\]

The relation for $\theta\in\left[  0,\pi/4\right]  ,$%
\[
\frac{-2\theta+4\pi}{5}<\pi-2\theta\Longleftrightarrow\theta\in\left[
0,\pi/8\right[  ,
\]
implies the existence of $e_{\theta}\in%
\mathbb{C}
$ ($\Re e_{\theta}<-1,\Im e_{\theta}>0$) as a unique intersection between
$\Sigma_{-1,\theta}^{l}$ and $\Sigma_{1,\theta}^{l}.$ If for some $\theta
\in\left[  \pi/8,\pi/4\right]  ,$ $\Sigma_{-1,\theta}^{l}$ and $\Sigma
_{1,\theta}^{l}$ meet, then they should do it at least twice because
$\Sigma_{\theta}$ diverges to $\infty$ in the direction $\pi-2\theta\in\left[
\pi/2,\frac{-2\theta+4\pi}{5}\right]  ;$ which violates the uniqueness of the
intersection if exists. Moreover, straightforward calculations show that the
set $\mathcal{E}$ is a smooth curve satisfying,%
\begin{align*}
\lim\limits_{\theta\longrightarrow0^{+}}e_{\theta}  &  =-1,\lim\limits_{\theta
\longrightarrow0^{+}}\arg\left(  e_{\theta}+1\right)  =\left(  \frac{3\pi}%
{4}\right)  ^{-};\\
\lim\limits_{\theta\longrightarrow\left(  \frac{\pi}{8}\right)  ^{-}}%
e_{\theta}  &  =\infty,\lim\limits_{\theta\longrightarrow\left(  \frac{\pi}%
{8}\right)  ^{-}}\arg\left(  e_{\theta}+1\right)  =\left(  \frac{3\pi}%
{4}\right)  ^{-}.
\end{align*}
See Fig.\ref{FIG00}.

The zero-level sets $\Sigma_{1,\theta}$ (resp.$\Sigma_{-1,\theta}$) of the
harmonic functions $\Re f_{1,\theta}$ (resp.$\Re f_{-1,\theta}$) split $%
\mathbb{C}
\setminus\left]  -\infty,-1\right]  $ (resp.$%
\mathbb{C}
\setminus\left[  1,+\infty\right[  $) into four disjoint connected domains. In
each domain, $\Re f_{1,\theta}$ (resp.$\Re f_{-1,\theta}$) preserves a fixed
sign, and in two adjacent domains, it has opposite signs. It follows that
$\Sigma_{1,\theta}\cup\Sigma_{-1,\theta}$ split the complex plane into a
finite number $n_{\theta}$ of disjoint connected domains $\Omega
_{k},k=1,...,n_{\theta}.$ In particular, the set $\Sigma_{\theta}$ is included
in those on which $\Re f_{1,\theta}$ and $\Re f_{-1,\theta}$ have the same sign.

Taking the arguments in $\left]  -\pi,\pi\right]  $ and the square roots such
that $\sqrt{1}=1,$ we have for $\theta\in\left]  0,\pi/4\right]  $ and $a\in%
\mathbb{C}
\setminus%
\mathbb{R}
$ $:$%
\begin{equation}
\frac{\arg\left(  a+1\right)  }{2}+\theta<\arg\left(  e^{i\theta}\int_{\left[
-1,1\right]  }\sqrt{p_{a}\left(  t\right)  }dt\right)  <\frac{\arg\left(
a-1\right)  }{2}+\theta. \label{strip sigma theta}%
\end{equation}
It follows that
\[
\frac{\Im a}{\Re\left(  a+1\right)  }<-\tan2\theta<\frac{\Im a}{\Re\left(
a-1\right)  }.
\]
If $\Im a<0,$ then (\ref{strip sigma theta}) gives
\[
-\frac{\pi}{2}<-\frac{\pi}{2}+\theta<\arg\left(  e^{i\theta}\int_{\left[
-1,1\right]  }\sqrt{p_{a}\left(  t\right)  }dt\right)  <\theta<\frac{\pi}{2};
\]
and then
\[
\Sigma_{\theta}\cap\left\{  a\in%
\mathbb{C}
:\Im a<0\right\}  =\emptyset.
\]
Thus,%
\[
\left(  \Sigma_{1,\theta}\cup\Sigma_{-1,\theta}\right)  \cap\left\{  a\in%
\mathbb{C}
:\Im a<0\right\}  =\emptyset.
\]

Finally, standard arguments of complex analysis show that, the functions%
\begin{align*}
t  &  \longmapsto\left(  \sqrt{p_{a}\left(  t\right)  }\right)  _{+}%
=\lim\limits_{s\longrightarrow t,\Im s>0}\sqrt{p_{a}\left(  s\right)  },\\
t  &  \longmapsto\left(  \sqrt{p_{a}\left(  t\right)  }\right)  _{-}%
=\lim\limits_{s\longrightarrow t,\Im s<0}\sqrt{p_{a}\left(  s\right)  }%
\end{align*}
are well defined and continuous in $\left[  -1,1\right]  .$ An elementary
study shows that in $\left[  -1,1\right]  ,$
\begin{align*}
a  &  \longmapsto\Re\int_{\left[  -1,1\right]  }\left(  \sqrt{p_{a}\left(
t\right)  }\right)  _{+}dt,\\
a  &  \longmapsto\Im%
{\displaystyle\int_{\left[  -1,1\right]  }}
\left(  \sqrt{p_{a}\left(  t\right)  }\right)  _{-}dt
\end{align*}
respectively, increases from $0$ to $\frac{16\sqrt{2}}{15}$ and decreases from
$\frac{16\sqrt{2}}{15}$ to $0.$ It follows that the function
\[%
\begin{array}
[c]{cc}%
\left[  -1,1\right]  \longrightarrow\left]  -\pi,\pi\right[  & a\longmapsto
\arg%
{\displaystyle\int_{\left[  -1,1\right]  }}
\left(  \sqrt{p_{a}\left(  t\right)  }\right)  _{-}dt
\end{array}
\]
decreases from $\frac{\pi}{2}$ to $0;$ which proves the existence and
uniqueness of $s_{\theta}$ in $\left[  -1,1\right]  .$ Relation
(\ref{relat symm}) shows that
\[
s_{\theta}\in\left\{
\begin{array}
[c]{c}%
\left[  -1,0\right[  \text{ if }\theta\in\left[  0,\frac{\pi}{4}\right[  ;\\
\left\{  0\right\}  \text{ if }\theta=\frac{\pi}{4};\\
\left]  0,1\right]  >0\text{ if }\theta\in\left]  \frac{\pi}{4},\frac{\pi}%
{2}\right]  .
\end{array}
\right.
\]
The observation that for $a\in%
\mathbb{C}
\setminus\left[  -1,1\right]  ,\int_{-1}^{1}\sqrt{\frac{1-t^{2}}{a-t}}%
dt\neq0,$ (as a period of complete elliptic integral) shows that the harmonic
function $f_{\theta}^{\prime}\left(  a\right)  $ defining $\Sigma_{\theta}$
has no critical points. If $\gamma$ is another arc in $\Sigma_{\theta},$ not
going through $t_{\theta},$ its two rays diverge differently to $s_{\theta}$
and to $\infty$ in direction $\arg a=\pi-2\theta,$ which violates the maximum
modulus principle. Thus, $\Sigma_{\theta}=\gamma.$ Notice that, if we
parameterize $\Sigma_{\theta}$ by $\left\{  a\left(  s\right)  ,s\geq
0\right\}  ,$ then it follows from (\ref{strip sigma theta}), and
(\ref{resigmathetapositif}) that
\begin{equation}
0\leq\arg\left(  \frac{a^{\prime}\left(  s\right)  }{a\left(  s\right)
+1}\right)  \leq\pi;\frac{-\pi}{2}\leq\arg\left(  \frac{a^{\prime}\left(
s\right)  }{a\left(  s\right)  -1}\right)  \leq0.
\label{monotony of sigmatheta}%
\end{equation}
Therefore, the functions $s\longmapsto\arg$ $\left(  a\left(  s\right)
+1\right)  ,s\longmapsto\arg$ $\left(  a\left(  s\right)  -1\right)  ,$ and
$s\longmapsto\arg$ $\left\vert a\left(  s\right)  -1\right\vert ,$
respectively increases from $0$ to $\pi-2\theta,$ decreases from $\pi$ to
$\pi-2\theta,$ and increases from $0$ to $+\infty.$
\end{proof}

\begin{figure}[tbh]
\begin{minipage}[b]{0.48\linewidth}
		\centering\includegraphics[scale=0.40]{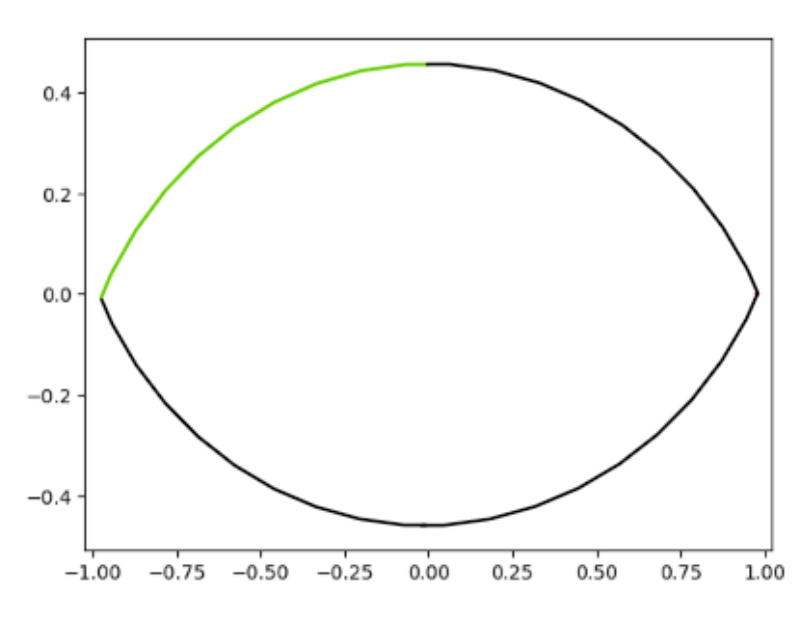}
	\end{minipage}\hfill
\begin{minipage}[b]{0.48\linewidth} \includegraphics[scale=0.4]{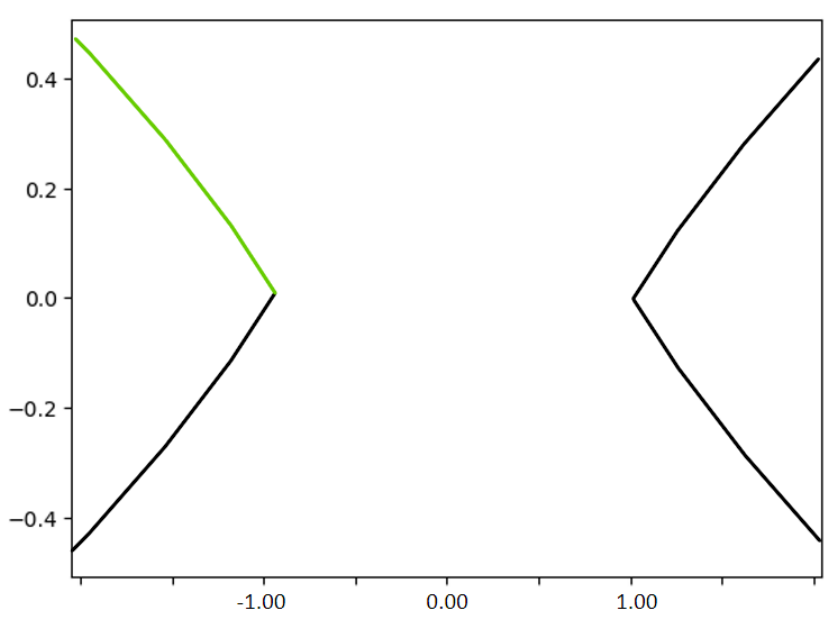}
	\end{minipage}\caption{Sets of $e_{\theta}$ (left) and $a_{\theta}$ (right).}%
\label{FIG00}%
\end{figure}

\begin{exercise}
\label{sets sigma}%
\[%
\begin{tabular}
[c]{|l|l|l|l|l|}\hline
$\diagdown\theta$ & $0$ & $\arctan\left(  1/2\right)  $ & $\pi/8$ & $\pi
/4$\\\hline
$t_{\theta}$ & $-1$ & $-0.69..+i0.25..$ & $-0.49..+i0.36..$ & $i0.462..$%
\\\hline
$s_{\theta}$ & $-1$ & $-0,37..$ & $-0.236..$ & $0$\\\hline
$e_{\theta}$ & $-1$ & $-1.67..+i0.79$ & none & none\\\hline
\end{tabular}
\ \ \ \ \ \ \ \ \ \
\]

\end{exercise}

See Fig.\ref{FIG1500}. \begin{figure}[tbh]
\begin{minipage}[b]{0.4\linewidth}
		\centering\includegraphics[scale=0.45]{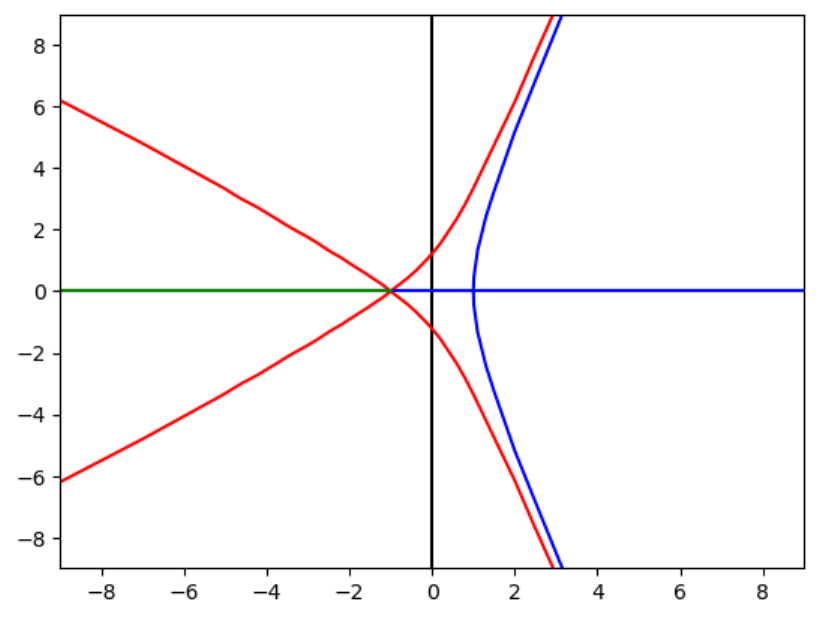}	
	\end{minipage}\hfill\begin{minipage}[b]{0.4\linewidth}
	\centering\includegraphics[scale=0.45]{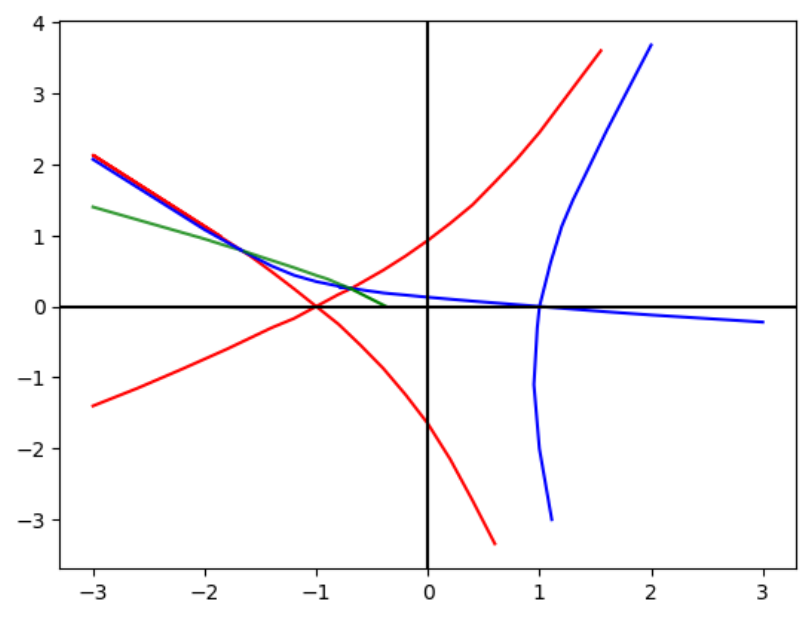}	
\end{minipage}\hfill\begin{minipage}[b]{0.4\linewidth}
		\centering\includegraphics[scale=0.45]{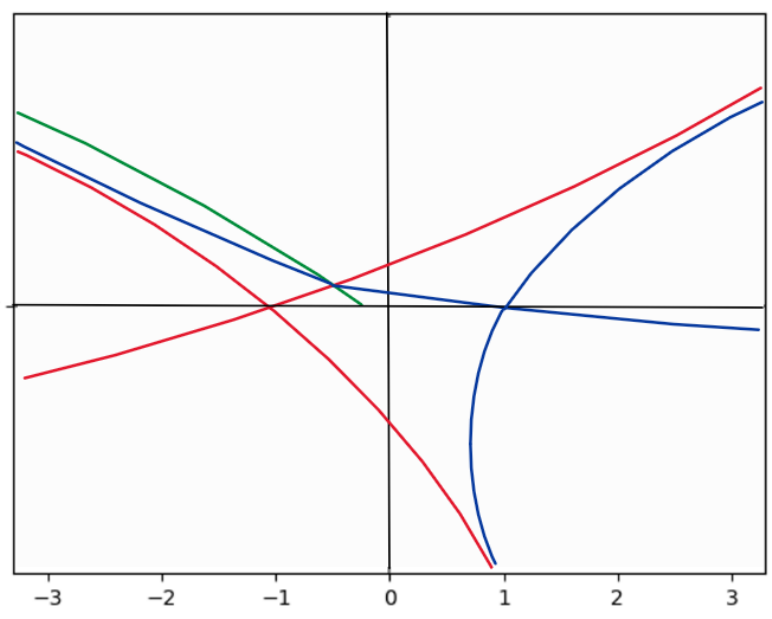}
	\end{minipage}\hfill\begin{minipage}[b]{0.4\linewidth}
		\centering\includegraphics[scale=0.45]{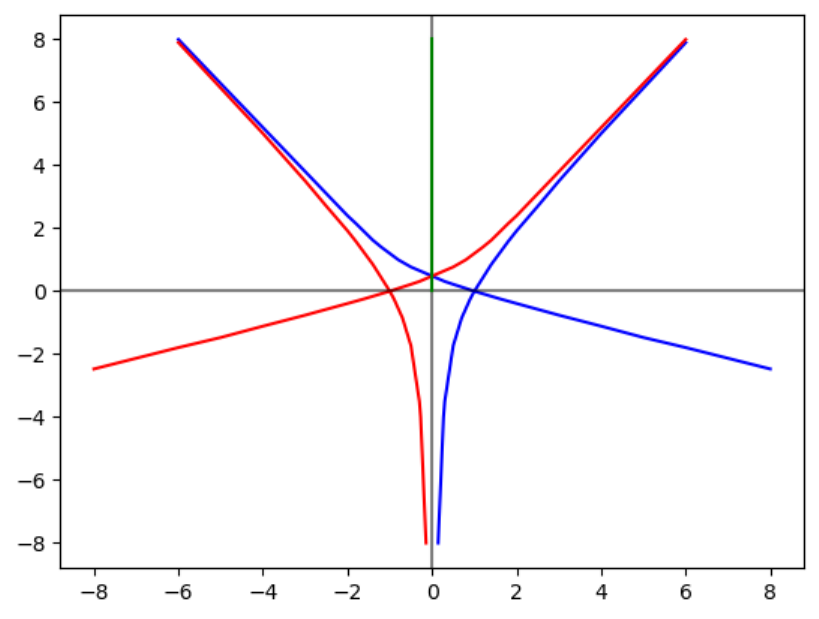}	
	\end{minipage}\hfill\caption{Approximate plots of the sets $\Sigma_{1,\theta
}$ (blue), $\Sigma_{-1,\theta}$ (red), and $\Sigma_{\theta}$ (green)
respectively for $\theta=0,$ $\theta=\arctan(0.5)/2,$ $\theta=\pi/8,$ and
$\theta=\pi/4$.}%
\label{FIG1500}%
\end{figure}

\section{A quadratic differential}

The theory of quadratic differentials appears in many areas of mathematics and
mathematical physics such as univalent functions and asymptotic theory of
linear ordinary differential equations, orthogonal polynomials, moduli spaces
of algebraic curves, and quantum mechanic. In this section, we consider the
quadratic differential on the Riemann sphere $\widehat{%
\mathbb{C}
},$%
\[
\varpi=\lambda p\left(  z\right)  dz^{2},
\]
where $p\left(  z\right)  $ is a $3$-degree complex polynomial with simple
zeros, and $\lambda$ is a non-vanishing complex number. A full description of
its critical graph can be found in \cite{maosero}, but regardless the location
of its zero.

By a linear change of variable, and taking into account that the critical
graph is invariant under multiplication of $\varpi$ by a positive real, we may
assume, without loss of the generality that
\[
\varpi=\varpi_{a,\theta}=-e^{2i\theta}p_{a}\left(  z\right)  dz^{2}%
=-e^{2i\theta}\left(  z-a\right)  \left(  z^{2}-1\right)  dz^{2},
\]
for some $a\in%
\mathbb{C}
\setminus\left\{  \pm1\right\}  ,$ and $\theta\in\left[  0,\pi\right[  .$ The
relations%
\begin{align}
\varpi_{a,\theta}  &  >0\Longleftrightarrow\varpi_{-a,\theta+\frac{\pi}{2}%
}>0,\label{sym1}\\
\varpi_{a,\theta}  &  >0\Longleftrightarrow\varpi_{-\overline{a},\frac{\pi}%
{2}-\theta}>0, \label{sym2}%
\end{align}
let us take the restriction $\theta\in\left[  0,\pi/4\right]  .$ More focus
will be on the two cases $\theta=0,$ and $\theta=\pi/4.$

In a first plan, we give some immediate and brief observations from the theory
of quadratic differentials. For more details, we refer the reader to
\cite{strebel},\cite{jenkins},\cite{jenk-spencer},...

Critical points of $\varpi_{a,\theta}$ are its zero's and poles in $\widehat{%
\mathbb{C}
}$. Zeros are called \emph{finite critical points}, while poles of order 2 or
greater are called \emph{infinite critical points}. All other points of
$\widehat{%
\mathbb{C}
}$ are called \emph{regular points}. Finite critical points of the quadratic
differential $\varpi_{a,\theta}$ are $\pm1$ and $a,$ as simple zeros; while,
by the change of variable $y=1/z,$ $\varpi_{a,\theta}$ has a unique infinite
critical point that is located at $\infty$ as a pole of order $7.$

\emph{Horizontal trajectories, }or \emph{Stokes lines }(or just trajectories,
in this section) of the quadratic differential $\varpi_{a,\theta}$ are the
level curves defined by
\begin{equation}
\Re\int^{z}e^{i\theta}\sqrt{p_{a}\left(  t\right)  }dt=\emph{const;}
\label{eq integ}%
\end{equation}
or equivalently%

\[
-e^{2i\theta}p_{a}\left(  z\right)  dz^{2}>0.
\]
If $z\left(  t\right)  ,t\in%
\mathbb{R}
$ is a horizontal trajectory, then the function
\[
t\longmapsto\Im\int^{t}e^{i\theta}\sqrt{p\left(  z\left(  s\right)  \right)
}z^{\prime}\left(  s\right)  ds
\]
is monotonic.

The \emph{vertical }(or, \emph{orthogonal}) \emph{trajectories} are obtained
by replacing $\Im$ by $\Re$ in equation (\ref{eq integ}). Horizontal and
vertical trajectories of the quadratic differential $\varpi_{a,\theta}$
produce two pairwise orthogonal foliations of the Riemann sphere $\widehat{%
\mathbb{C}
}.$ A trajectory passing through a critical point of $\varpi_{a,\theta}$ is
called \emph{critical trajectory}. In particular, if it starts and ends at a
finite critical point, it is called \emph{finite critical trajectory} or
\emph{short trajectory}, otherwise, we call it an \emph{infinite critical
trajectory}. The closure of the set of finite and infinite critical
trajectories, that we denote by $\Gamma_{a,\theta}$ is called the
\emph{critical }(or \emph{Stokes} ) \emph{graph}. The local and global
structure of the trajectories are studied in \cite{strebel},\cite{jenkins}%
,...etc (see Fig.\ref{FIG0}). At any regular point, horizontal (resp.
vertical) trajectories look locally as simple analytic arcs passing through
this point, and through every regular point of $\varpi_{a,\theta}$ passes a
uniquely determined horizontal (resp. vertical) trajectory of $\varpi
_{a,\theta};$ the are locally orthogonal at this point. From each zero with
multiplicity $r$ of $\varpi_{a,\theta},$ there emanate $r+2$ critical
trajectories spacing under equal angles $2\pi/(r+2).$ Since $\infty$ is a pole
of order $7,$ there are $5$ asymptotic directions (called \emph{critical
directions}) spacing under equal angle $2\pi/5,$ and a neighborhood
$\mathcal{V}$ of $\infty,$ such that each trajectory entering $\mathcal{V}$
stays in $\mathcal{V}$ and tends to $\infty$ following one of the critical
directions:
\begin{equation}
D_{\theta,k}=\left\{  z\in%
\mathbb{C}
:\arg\left(  z\right)  =\frac{-2\theta+\left(  2k+1\right)  \pi}{5}\right\}
;k=0,...,4. \label{critcal directions}%
\end{equation}
Analogues behavior happens to the orthogonal trajectories at $\infty$, but
their directions are :%
\[
D_{\theta,k}^{\perp}=\left\{  z\in%
\mathbb{C}
:\arg\left(  z\right)  =\frac{-2\theta+2k\pi}{5}\right\}  ;k=0,...,4.
\]
Observe that for $\theta\in\left[  0,\frac{\pi}{2}\right[  ,$ exactly three
critical directions, $D_{\theta,0},D_{\theta,1},$ and $D_{\theta,2}$ are in
the upper half-plane. \begin{figure}[tbh]
\begin{minipage}[b]{0.48\linewidth}
		\centering\includegraphics[scale=0.30]{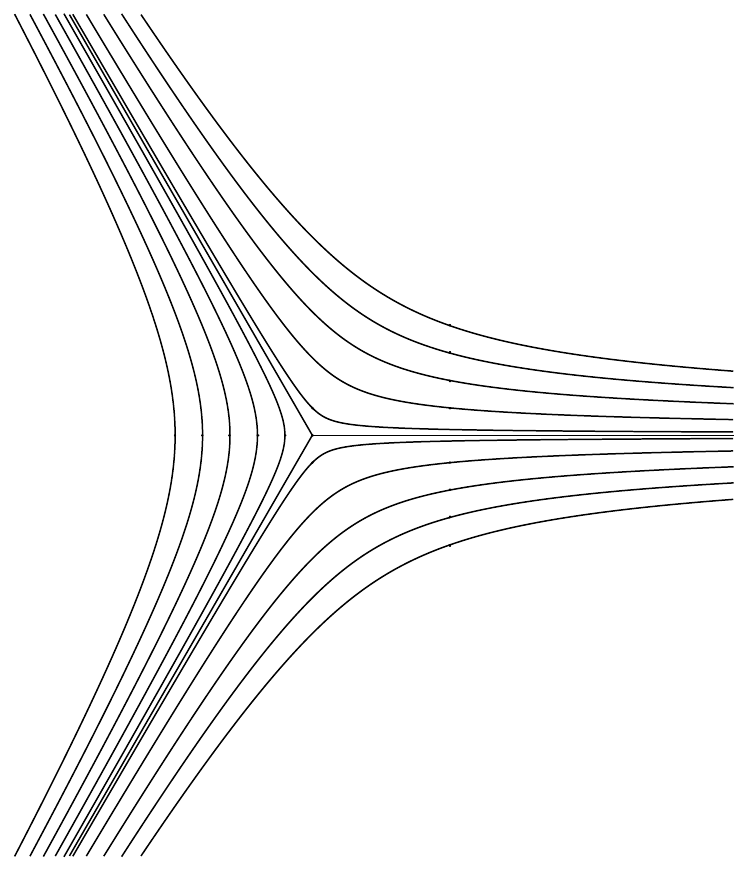}
	\end{minipage}\hfill
\begin{minipage}[b]{0.48\linewidth} \includegraphics[scale=0.4]{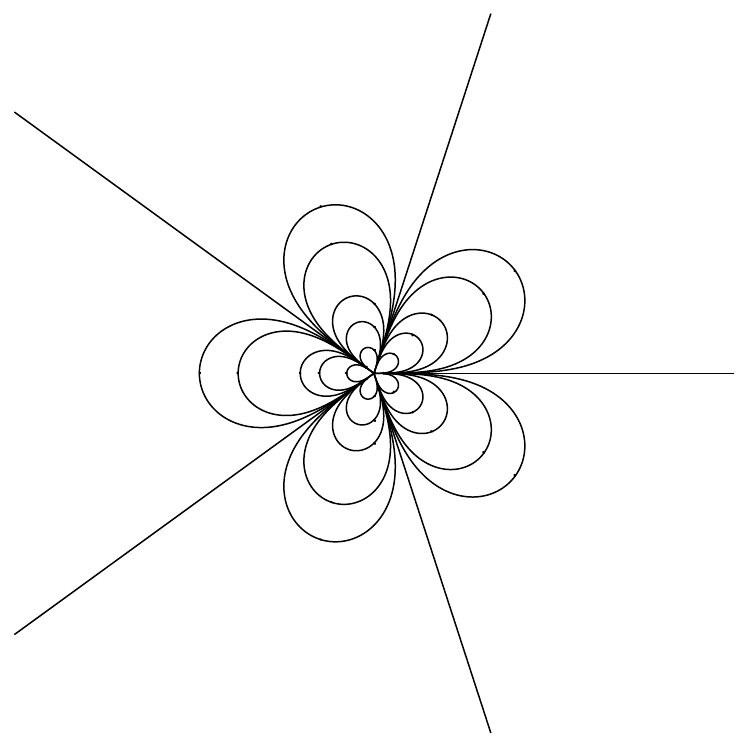}
	\end{minipage}
\caption{Local structure of the trajectories near a simple zero (left) and a
pole of order $7$ (right).}%
\label{FIG0}%
\end{figure}

In the large, any trajectory of $\varpi_{a,\theta}$ is either an arc
connecting two critical points, or, is an arc that diverges to $\infty$ at
least in one of its rays. Thanks to Jenkins Three pole Theorem, the quadratic
differential $\varpi_{a,\theta}$ has no recurrent trajectory ( trajectory that
is dense in some domain of $%
\mathbb{C}
$). The structure of the set $\widehat{%
\mathbb{C}
}\setminus\Gamma_{a,\theta}$ depends on the local and global behaviors of
trajectories. It consists of a finite number of domains called\emph{ domain
configurations} of $\varpi_{a,\theta}.$ Jenkins Theorem (see \cite[Theorem3.5]%
{jenkins},\cite{strebel}) asserts that there are two kinds of domain
configurations of $\varpi_{a,\theta}$ :

\begin{itemize}
\item Half-plane domain : It is swept by trajectories diverging to $\infty$ in
its two ends, and along consecutive critical directions. Its boundary consists
of a path with two unbounded critical trajectories, and possibly a finite
number of short ones. It is conformally mapped to a vertical half plane
$\left\{  w\in%
\mathbb{C}
:\Re w>c\right\}  $ for some real $c$ by the function $\int_{z_{0}}%
^{z}e^{i\theta}\sqrt{p_{a}\left(  t\right)  }dt$ with suitable choices of
$z_{0}\in%
\mathbb{C}
$ and the branch of the square root;

\item strip domain : It is swept by trajectories which both ends tend
$\infty.$ Its boundary consists of a disjoint union of two paths, each of them
consisting of two unbounded critical trajectories diverging to $\infty,$ and
possibly a finite number of short trajectories. It is conformally mapped to a
vertical strip $\left\{  w\in%
\mathbb{C}
:c_{1}<\Re w<c_{2}\right\}  $ for some reals $c_{1}$ and $c_{2}$ by the
function $\int_{z_{0}}^{z}e^{i\theta}\sqrt{p_{a}\left(  t\right)  }dt$ with
suitable choices of $z_{0}\in%
\mathbb{C}
$ and the branch of the square root.
\end{itemize}

A necessary condition for the existence of a short trajectory connecting two
finite critical points of $\varpi_{a,\theta}$ is the existence of a Jordan arc
$\gamma$ connecting them, such that
\begin{equation}
\Re\int_{\gamma}e^{i\theta}\sqrt{p_{a}\left(  t\right)  }dt=0.
\label{cond necess}%
\end{equation}
However, this condition is not sufficient in general as we will see in
theorem.\ref{main thm}. A powerful tool which will be used in our
investigation is the following lemma :

\begin{lemma}
[Teichm\"{u}ller ]\label{teich lemma}Let $\mathcal{C}$ be a closed curve
formed by a finite number of parts of horizontal and/or vertical trajectories
of $\varpi_{a,\theta}$ (with their endpoints). Let $z_{j}$ be the critical
points on $\mathcal{C},$ $\theta_{j}$ be the corresponding interior angles at
vertices $z_{j},$ and $n_{j}$ be the multiplicities of $z_{j}.$Then%
\begin{equation}
\sum_{j}\left(  1-\frac{\left(  n_{j}+2\right)  \theta_{j}}{2\pi}\right)
=2+\sum_{i}m_{i} \label{teich equality}%
\end{equation}
where $m_{i}$ are the multiplicities of critical points inside $\mathcal{C}.$
\end{lemma}

\begin{remark}
As an immediate observation from lemma \ref{teich lemma} is that
$\varpi_{a,\theta}$ has at most two short trajectories connecting two
different pairs of distinct finite critical points.
\end{remark}

\begin{remark}
\label{contra}Let $\Im a>0$ and $\alpha,\beta,\gamma\in%
\mathbb{R}
,\alpha<\beta\leq\gamma<\delta.$ Then:
\[
0<\arg\int_{\gamma}^{\delta}\sqrt{p_{a}\left(  t\right)  }dt<\arg\int_{\alpha
}^{\beta}\sqrt{p_{a}\left(  t\right)  }dt<\frac{\pi}{2},
\]
where the arguments are taken in $\left[  -\pi,\pi\right[  ,$ the square roots
such that $\sqrt{1}=1,$ and it is straightforward that the two properties
\begin{align}
\Re\left(  e^{i\theta}\int_{\alpha}^{\beta}\sqrt{p_{a}\left(  t\right)
}dt\right)   &  =0,\label{p1}\\
\Re\left(  e^{i\theta}\int_{\gamma}^{\delta}\sqrt{p_{a}\left(  t\right)
}dt\right)   &  =0 \label{p2}%
\end{align}
cannot hold together.
\end{remark}

\begin{lemma}
\label{lemme1}For any $a\in%
\mathbb{C}
\setminus%
\mathbb{R}
,$ a short trajectory of $\varpi_{a,\theta}$ if exists, it is entirely
included in the half plane bordered by $%
\mathbb{R}
$ and containing $a.$
\end{lemma}

\begin{proof}
Suppose that a short trajectory of $\varpi_{a,\theta}$ connects $z=1$ to
$z=a,\Im a>0,$ going through some point $b$ in the real line, for example
$b<-1.$ If a critical trajectory of $\varpi_{a,\theta}$ cuts the real line in
some point $c,$ then, putting $\alpha=c,\beta=b,$ and $\gamma=1,$ we get the
same contradiction as in remark.\ref{contra}. If not, the two remaining
critical trajectories should diverge to $\infty$ in the directions
$D_{\theta,1}$ and $D_{\theta,2}.$ Lemma.\ref{teich lemma} insures that two
remaining critical trajectories emerging from $z=1$ should diverge to $\infty$
in the lower half plane, one of them, cuts the real line in some $\gamma>1.$
Again, $\alpha=b,\beta=1,$ and $\gamma$ give the same contradiction.
\end{proof}

\begin{notation}
\label{sets nota}For $\theta\in\left[  0,\pi/4\right]  ,$ we denote by

\begin{itemize}
\item $S_{\pm1,\theta}:$ the sets $\Sigma_{\pm1,\theta}$ minus the arc
starting at $e_{\theta}$ (if exists), and diverging to $\infty$ in the upper
half plane;

\item $S_{\theta}:$ the part of $\Sigma_{\theta}$ starting at $t_{\theta}$ and
diverging to $\infty;$

\item $\Xi_{\theta}:S_{1,\theta}\cup S_{-1,\theta}\cup S_{\theta;}$

\item $n_{\theta}:$ the finite number ($n_{0}=8,n_{arct\left(  0.5\right)
/2}=10,$and $n_{\pi/4}=9$) of the connected components $\Omega_{1}%
,...,\Omega_{n_{\theta}}$ of $%
\mathbb{C}
\setminus\left(  S_{\pm1,\theta}\cup S_{\theta}\right)  .$ See
Fig.\ref{FIG900}
\end{itemize}
\end{notation}

\begin{figure}[tbh]
\begin{minipage}[b]{0.48\linewidth}
		\centering\includegraphics[scale=0.28]{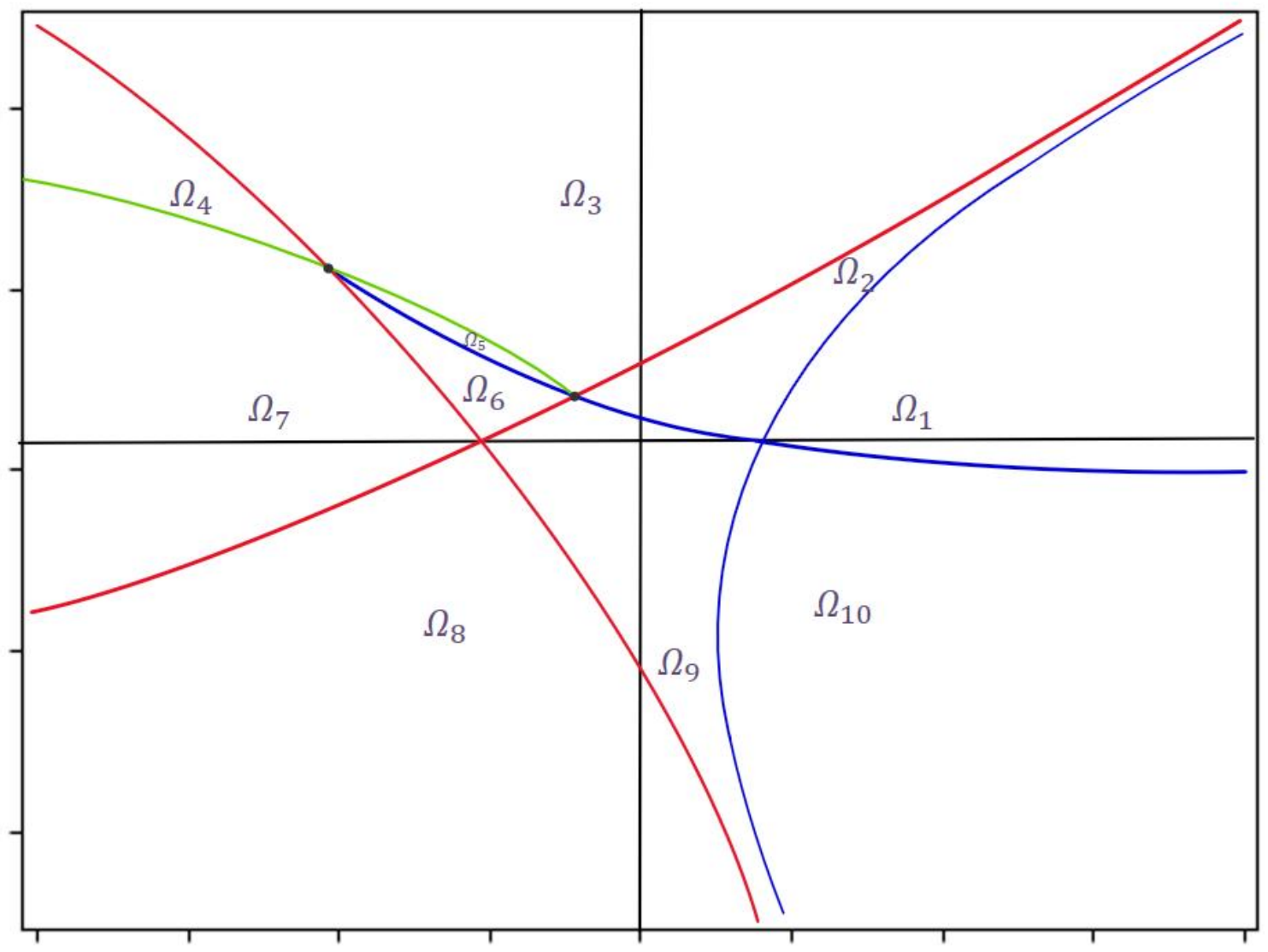}
	\end{minipage}\hfill
\begin{minipage}[b]{0.48\linewidth} \includegraphics[scale=0.28]{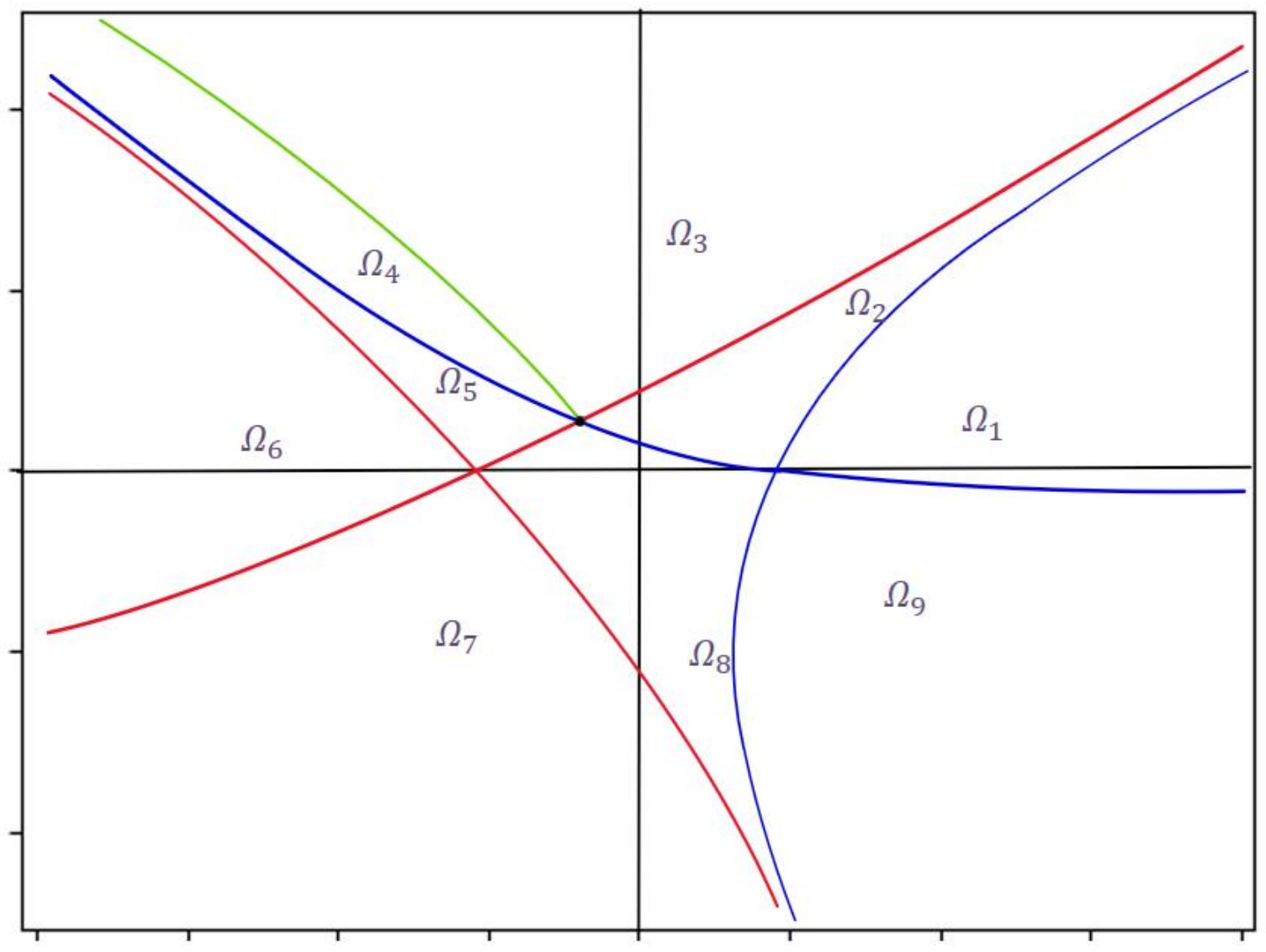}
	\end{minipage}
\caption{Sets $S_{\pm1,\theta}$ and $S_{\theta}$ for $\theta\in\left[
0,\pi/8\right[  $ (left) and $\theta\in\left[  \pi/8,\pi/4\right]  $ (right).}%
\label{FIG900}%
\end{figure}

The main result of this section is the following :

\begin{theorem}
\label{main thm}Let $\theta\in\left[  0,\pi/2\right[  .$ In each of the
domains $\Omega_{1},...,\Omega_{n_{\theta}},$ the critical graph
$\Gamma_{a,\theta}$ of the quadratic differential $\varpi_{a,\theta}$ has the
same structure; it splits the Riemann sphere into five half-plane and two
strip domains. Moreover, $\varpi_{a,\theta}$

\begin{itemize}
\item has a short trajectory connecting $z=\pm1$ if and only if $a\in
S_{\theta};$ it has a short trajectory connecting $z=\pm1$ to $z=a$ if and
only if $a\in S_{\pm1,\theta}.$ In all these cases, $\Gamma_{a,\theta}$ splits
the Riemann sphere into five half-plane domains and exactly one strip domain;

\item It has a tree (juxtaposition of two short trajectories) with summit
$t_{\theta};$ it has a tree with summit $z=-1$ or $z=+1$ respectively for
$\theta\in$ $\left[  0,\pi/8\right[  $ and $\theta\in$ $\left[  3\pi
/8,\pi/2\right[  .$ In these cases, $\Gamma_{a,\theta}$ splits the Riemann
sphere into five half-plane domains.
\end{itemize}

In particular, any change of a critical graph structure should pass by a
critical graph with at least one short trajectory.
\end{theorem}

\begin{proof}
We begin by the observation that if $\gamma$ is a short trajectory of
$\varpi_{a,\theta}$ connecting $\pm1,$ then $z=a$ belongs to the part of
$\Sigma_{\theta}$ included in the exterior of the open bounded domain bordered
by $\gamma\cup\left[  -1,1\right]  ;$ otherwise, all critical trajectories
emerging from $a$ should cut the segment $\left[  -1,1\right]  ,$ and we get
the same contradiction as in the proof of lemma.\ref{lemme1}. It follows that
$\gamma$ is homotopic to $\left[  -1,1\right]  $ in $%
\mathbb{C}
\setminus\left\{  a\right\}  ,$ and then $a\in\Sigma_{\theta}.$ Let
parameterize $\Sigma_{\theta},$ for $\theta\in\left]  0,\frac{\pi}{4}\right]
,$ by :
\[
\Sigma_{\theta}=\left\{  a=a\left(  s\right)  ;s\geq0\right\}  \subset\left\{
z\in%
\mathbb{C}
:\Re z\leq0,\Im z\geq0\right\}  ,a\left(  0\right)  =s_{\theta}.
\]
Since $\left]  -1,1\right[  \cap\Sigma_{\pm1,\theta}=\emptyset,$ the quadratic
differential $\varpi_{s_{\theta},\theta}$ has no short trajectory connecting
$\pm1;$ more precisely, if $\varpi_{s_{\theta},\theta}$ has only one critical
trajectory emerging from $z=s_{\theta}$ that diverges to $\infty$ in lower
half-plane in a direction $D_{\theta,k},k\in\left\{  3,4\right\}  ,$ then,
lemma.\ref{teich lemma} asserts that two critical trajectories $\gamma_{1}$
and $\gamma_{-1}$ emerging respectively from $z=1$ and $z=-1,$ should diverge
to $\infty$ in the same direction $D_{\theta,k}.$ From the local behavior of
the trajectories at $\infty,$ there exists a neighborhood $\mathcal{V}$ of
$\infty,$ such that any vertical trajectory in $\mathcal{V}$ has its two rays
diverging to $\infty$ following the two adjacent critical directions
$D_{\theta,k}^{\perp}$ and $D_{\theta,\left(  k+1\right)  \operatorname{mod}%
(5)}^{\perp}.$ Let be $\sigma$ such a vertical trajectory. It is obvious that
$\sigma$ intersects $\gamma_{1}$ and $\gamma_{-1}$ respectively in two points
$z=m_{1}$ and $z=m_{-1}.$ Consider the path $\gamma$ connecting $z=1$ and
$z=-1$ formed by, the part of $\gamma_{1}$ from $z=1$ to $z=m_{1},$ the part
of $\sigma$ from $z=m_{1}$ to $z=m_{-1},$ and the part of $\gamma_{-1}$ from
$z=m_{-1}$ to $z=-1;$ see Fig.\ref{FIG0033} (left). \begin{figure}[tbh]
\begin{minipage}[b]{0.4\linewidth}
		\centering\includegraphics[scale=0.465]{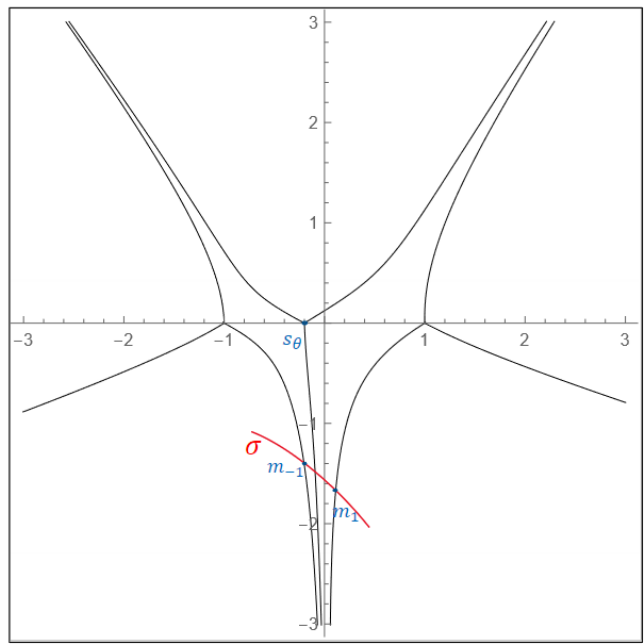}
	\end{minipage}\hfill
\begin{minipage}[b]{0.4\linewidth} \includegraphics[scale=0.5]{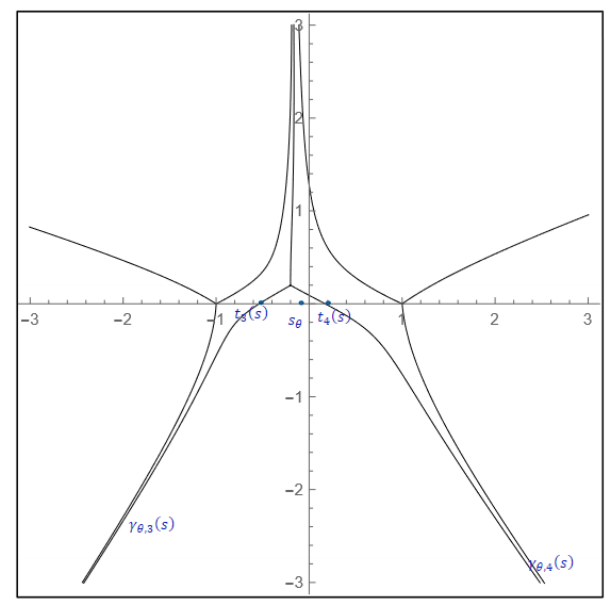}
	\end{minipage}
\caption{Path $\gamma$ (left), $\gamma_{\theta,3}\left(  s\right)  $ and
$\gamma_{\theta,4}\left(  s\right)  $ (right).}%
\label{FIG0033}%
\end{figure}Integrating along $\gamma,$ we get
\begin{align}
\Re\int_{\gamma}e^{i\theta}\sqrt{p_{a}\left(  t\right)  }dt  &  =\Re\int
_{1}^{m_{1}}e^{i\theta}\sqrt{p_{a}\left(  t\right)  }dt+\Re\int_{m_{1}%
}^{m_{-1}}e^{i\theta}\sqrt{p_{a}\left(  t\right)  }dt+\Re\int_{m_{-1}}%
^{-1}e^{i\theta}\sqrt{p_{a}\left(  t\right)  }dt\label{path}\\
&  =\Re\int_{m_{1}}^{m_{-1}}e^{i\theta}\sqrt{p_{a}\left(  t\right)  }%
dt\neq0;\nonumber
\end{align}
which violates (\ref{cond necess}); more precisely
\[
\Re\int_{\gamma}e^{i\theta}\sqrt{p_{a}\left(  t\right)  }dt=\Re\int_{\left[
-1,1\right]  }e^{i\theta}\left(  \sqrt{p_{a}\left(  t\right)  }\right)
_{-}dt=0.
\]
Thus, critical trajectories of $\varpi_{s_{\theta},\theta},$ emerging from
$s_{\theta},$ say $\gamma_{\theta,1}\left(  0\right)  ,\gamma_{\theta
,3}\left(  0\right)  $ and $\gamma_{\theta,4}\left(  0\right)  ,$ diverge to
$\infty$ respectively in the directions $D_{\theta,1},D_{\theta,3}$ and
$D_{\theta,4}.$ Moreover, by continuity (in the Haussd\"{o}rff metric) of the
critical graph $\Gamma_{s_{\theta},\theta}$ with respect to $s\geq0,$ there
exists $s_{0}>0,$ ($s_{0}$ is taken maximum) such that, for $0<s<s_{0},$
$\gamma_{\theta,3}\left(  s\right)  $ and $\gamma_{\theta,4}\left(  s\right)
$ intercept the real line respectively in $t_{3}\left(  s\right)  \in\left]
-1,s_{\theta}\right[  $ and $t_{4}\left(  s\right)  \in\left]  s_{\theta
},1\right[  ;$ see Fig.\ref{FIG0033} (right). Clearly, if $\gamma_{\theta
,3}\left(  s\right)  $ (resp. $\gamma_{\theta,4}\left(  s\right)  $) changes
its behavior for some $s>0,$ then it should go through $z=-1$ for some
$s_{-1}>0,$ (resp. $s_{1}>0$) and necessarily, $a\left(  s_{-1}\right)
\in\Sigma_{-1,\theta}$ (resp. $a\left(  s_{1}\right)  \in\Sigma_{1,\theta}$).
These facts, besides the same idea of the choice of paths $\gamma$ joining
$z=a$ to $\pm1,$ lead us to prove that, $a\left(  s_{0}\right)  =t_{\theta},$
$\varpi_{t_{\theta},\theta}$ has a tree with summit $t_{\theta},$ and,
$\varpi_{a,\theta}$ has no short trajectory for any $a\in\Omega_{8}\cap
\Sigma_{\theta}.$

Conversely, let $s>s_{0}$ and suppose that $\varpi_{a\left(  s\right)
,\theta}$ has no short trajectory connecting $\pm1.$ The critical graph of
$\varpi_{a\left(  s\right)  ,\theta}$ cannot be the same as in
Fig.\ref{FIG0033} (right), because, since $a\in\Omega_{3}\cup\Omega_{4},$
$\Re\int_{-1}^{a}e^{i\theta}\sqrt{p_{a}\left(  t\right)  }dt$ and $\Re\int
_{1}^{a}e^{i\theta}\sqrt{p_{a}\left(  t\right)  }dt$ both change their signs,
and necessarily we get a strip domain containing $\left]  -1,x\right[  $ for
some $x\in$ $\left]  -1,1\right]  ,$ that is bordered by two critical
trajectories emerging from $z=\pm1.$ Integrating along a path as in
(\ref{path}), we get a similar contradiction.

Now, for $\theta\in\left]  0,\pi/8\right[  ,$ we consider three reals
$u,v_{1},$ and $v_{2}$ such that
\[
u<\Re e_{\theta}<-1;a_{1}=u+iv_{1}\in\Sigma_{\theta};a_{2}=u+iv_{2}\in
\Sigma_{-1,\theta}^{l}.
\]
Let%
\[
a=u+iv,v\in\left]  0,+\infty\right[  .
\]%
\[
a=u+iv=\left\vert a+1\right\vert e^{i\arg\left(  a+1\right)  },v\in\left]
v_{1},v_{2}\right[  ,\arg\left(  a+1\right)  \in\left]  \frac{\pi}{2}%
,\pi\right[  .
\]
We denote by $\gamma_{k}\left(  t\right)  ,k=1,2,3,$ the three critical
trajectories emerging from $z=-1$ :
\[
\gamma_{k}\left(  t\right)  +1=r_{k}\left(  t\right)  e^{i\alpha_{k}\left(
t\right)  },r_{1}\left(  t\right)  >0,\alpha_{k}\left(  t\right)  \in\left[
\frac{-\pi}{2},\frac{3\pi}{2}\right[  ,t>0.
\]
Straightforward calculations give :%
\[
\int_{0}^{t}\sqrt{p_{a}\left(  \gamma_{1}\left(  s\right)  \right)  }%
\gamma_{1}^{\prime}\left(  s\right)  ds\sim\frac{2\sqrt{2}}{3}\left(
\sqrt{a+1}\left(  \gamma_{1}\left(  t\right)  +1\right)  ^{\frac{3}{2}%
}\right)  ,t\longrightarrow0^{+},
\]
and
\[
\arg\left(  e^{i\theta}\sqrt{a+1}\left(  \gamma_{k}\left(  t\right)
+1\right)  ^{\frac{3}{2}}\right)  \equiv\frac{-\pi}{2}\operatorname{mod}%
\left(  \pi\right)  ,t\longrightarrow0^{+}.
\]
Then
\[
\alpha_{k}\left(  0^{+}\right)  =\frac{-\arg\left(  a+1\right)  -2\theta
+\left(  2k-1\right)  \pi}{3};k\in\left\{  1,2,3\right\}  .
\]
Since $\arg\left(  a+1\right)  \in\left]  \frac{\pi}{2},\pi\right[  ,$ we get%
\[
\alpha_{1}\left(  0^{+}\right)  \in\left]  \frac{-\pi}{12},\frac{\pi}%
{6}\right[  ;\alpha_{2}\left(  0^{+}\right)  \in\left]  \frac{7\pi}{12}%
,\frac{5\pi}{6}\right[  ;\alpha_{3}\left(  0^{+}\right)  \in\left]  \frac
{5\pi}{4},\frac{3\pi}{2}\right[  .
\]
In other words, $\gamma_{2}\left(  t\right)  $ and $\gamma_{3}\left(
t\right)  $ rise respectively in the upper and lower half planes, while the
argument $\alpha_{1}\left(  0^{+}\right)  $ of $\gamma_{1}\left(
0^{+}\right)  +1$ increases from $\frac{-\pi}{12}$ to $\frac{\pi}{6}$ as $v$
increases from $0$ to $+\infty.$ The same reasoning shows that from $z=1$
emerges a critical trajectory $\delta_{1}\left(  t\right)  $ such that, the
argument $\beta_{1}\left(  0^{+}\right)  $ of $\delta_{1}\left(  0^{+}\right)
+1$ decreases from $\frac{7\pi}{6}$ to $\frac{11\pi}{12}$ as $v$ increases
from $0$ to $+\infty.$ It is obvious that $\delta_{1}\left(  t\right)  $ is
the unique eligible critical trajectory emerging from $z=1$ to be finite
(ending at $z=a$ ).

For $a=a_{1},\Re\int_{-1}^{1}e^{i\theta}\sqrt{p_{a_{1}}\left(  t\right)
}dt=0,$ which gives $b\in\left]  -1,1\right[  $ such that%
\[
\left(  \Re\int_{-1}^{b}e^{i\theta}\sqrt{p_{a_{1}}\left(  t\right)
}dt\right)  \left(  \Re\int_{b}^{1}e^{i\theta}\sqrt{p_{a_{1}}\left(  t\right)
}dt\right)  <0.
\]
With the same choice of the square roots as in above, it is obvious that
\[
\arg\sqrt{p_{a}\left(  r\right)  }<\arg\sqrt{p_{a}\left(  t\right)  }%
<\arg\sqrt{p_{a}\left(  s\right)  },\forall r\in\left]  -1,b\right[  ,\forall
s\in\left]  b,1\right[  ,
\]
and then%
\begin{equation}
\Re\int_{b}^{1}e^{i\theta}\sqrt{p_{a_{1}}\left(  t\right)  }dt<0.
\label{INEQ b}%
\end{equation}
By continuity, inequality (\ref{INEQ b}) holds for $a=u+iv,$ $v\in\left]
v_{1},v_{1}+\varepsilon\right[  ,$ for some $\varepsilon>0.$ Since%
\[
\Re\int_{-1}^{1}e^{i\theta}\sqrt{p_{a}\left(  t\right)  }dt>0,\forall\Im
a>v_{1},
\]
a quick study of the function $x\longmapsto\Re\int_{x}^{1}e^{i\theta}%
\sqrt{p_{a}\left(  t\right)  }dt$ defined in $\left[  -1,1\right]  ,$ shows
that it vanishes in some unique point $x_{a}\in$ $\left]  -1,1\right[  .$ It
follows that the quadratic differential $\varpi_{a,\theta}$ has a critical
trajectory (more precisely, $\delta_{1}\left(  t\right)  $ ) that intersects
the real line in $x_{a}.$ With remark.\ref{contra}, $\delta_{1}\left(
t\right)  $ should diverge to $\infty$ in the lower half plane. The same
remark proves that $\delta_{1}\left(  t\right)  $ is the only critical
trajectory emerging from $z=1$ that can go through $z=a.$ If for some $v\geq
v_{1}+\varepsilon,$ $\delta_{1}\left(  t\right)  $ stays in the upper half
plane, then, by continuity of the trajectories in the Haussd\"{o}rf metric, it
should meet $\gamma_{1}\left(  t\right)  $ and $\varpi_{a,\theta}$ has a short
trajectory connecting $\pm1$ for an $a\notin\Sigma_{\theta}$ which is
impossible. Thus, for any $v>v_{1},$ $\varpi_{a,\theta}$ has no short
trajectory connecting $z=1$ and $z=a.$

From lemma.\ref{teich lemma}, acknowledging the behavior of critical
trajectories of $\varpi_{a,\theta}$ emerging from two zeros, gives a full idea
about the critical graph $\Gamma_{a,\theta}.$ For $i\in\left\{
1,...,n_{\theta}\right\}  ,$ we consider the equivalence relation
$\mathcal{R}_{i}$ in $\Omega_{i}$ :
\[
a\mathcal{R}_{i}a^{\prime}\text{if and only if, critical graphs}%
\Gamma_{a,\theta}\text{ and }\Gamma_{a^{\prime},\theta}\text{ have the same
structure.}%
\]
In other words, for any critical trajectory of $\varpi_{a,\theta}$ emerging
from $z=-1,$ (resp. $z=+1$), there exists a correspondent critical trajectory
of $\varpi_{a^{\prime},\theta}$ that emerges from $z=-1$ (resp. $z=+1$), and
diverges to $\infty$ following the same critical direction. Let be
$\Lambda_{1},...,\Lambda_{m},m\geq1,$ the equivalence classes of
$\mathcal{R}_{i},$ and $b_{n}$ a sequence of points in $\Lambda_{j},1\leq
j\leq m,$ converging to some $b\in\Omega_{i}.$ It is obvious, with an adequate
choice of the square roots, that the sequence of functions
\[
f_{n}\left(  z\right)  =e^{i\theta}\sqrt{\left(  b_{n}-z\right)  \left(
1-z^{2}\right)  }%
\]
converges uniformly on compact sub-sets of $\Omega_{i}$ to the function
\[
f\left(  z\right)  =e^{i\theta}\sqrt{\left(  b-z\right)  \left(
1-z^{2}\right)  }.
\]
Therefore, if, for example, $\gamma_{-1,b_{n}}$ is a critical trajectory of
$\varpi_{b_{n},\theta}$ emerging from $z=-1,$ and diverging to $\infty$ in a
direction $D_{\theta,k},$ then, it should converge (in the Haussd\"{o}rf
metric) to a correspondent critical trajectory $\gamma_{-1,b}$ of
$\varpi_{b,\theta}$ diverging to $\infty$ in the same direction. Indeed,
consider a sequence $\left(  y_{n}\right)  \subset\gamma_{-1,b_{n}},$ and
converging to some $y\in%
\mathbb{C}
.$ It is obvious then that :%
\[
\Re\left(  e^{i\theta}\int_{-1}^{y_{n}}\sqrt{\left(  b_{n}-t\right)  \left(
1-t^{2}\right)  }dt\right)  =\Re\left(  e^{i\theta}\int_{-1}^{y}\sqrt{\left(
b-t\right)  \left(  1-t^{2}\right)  }dt\right)  =0,
\]
and $y\in\gamma_{-1,b}.$ Moreover, if $\gamma_{-1,b}$ diverges to $\infty,$
then it cannot follow a direction different to $D_{\theta,k},$ unless he meets
another critical trajectory, and we get a short one, which cannot hold since
$b\in\Omega_{i}.$ Hence, $b_{n}\mathcal{R}_{i}b,$ and, $\Lambda_{j}$ is a
closed sub-set of $\Omega_{i}.$ Since $\Omega_{i}$ is a non-empty connected
set, it cannot be a disjoint finite union of non-empty closed sub-sets. Thus,
$m=1,$ and all critical graphs $\left(  \Gamma_{a,\theta}\right)  _{a\in
\Omega_{i}}$ have the same structure.
\end{proof}

\subsection{\bigskip The case $\theta=0$}

Recall that for the case $\theta=0,$ $\Xi_{0}=\Sigma_{\pm1,0}\cup\Sigma
_{\pm1,0},$ and $n_{0}=8.$ For the sake of simplicity, we will restrict our
investigation in this section in the case of $\Xi_{0}\cap%
\mathbb{C}
^{+},$ where $%
\mathbb{C}
^{+}=\left\{  a\in%
\mathbb{C}
:\Im a>0\right\}  .$

\begin{lemma}
\label{intersects}Let $a\in%
\mathbb{C}
^{+}.$ Then,

\begin{itemize}
\item no short trajectory of $\varpi_{a,0}$ connects $z=1$ and $z=-1;$

\item no critical trajectory of the quadratic differential $\varpi_{a,0}$
emerging from $z=\pm1$ intersects $%
\mathbb{R}
\setminus\left\{  \pm1\right\}  ;$

\item there exist at most one critical trajectory of $\varpi_{a,0}$ emerging
from $z=a,$ intersecting the real line; if it exists, the intersection is in a
unique point;

\item for $\Re a\leq1,$ there is no critical trajectory of $\varpi_{a,0}$
emerging from $z=1$ and intersecting $i%
\mathbb{R}
^{+}.$
\end{itemize}
\end{lemma}

\begin{proof}
Let be $\gamma$ a critical trajectory of $\varpi_{a,0}$ emerging for example
from $z=1$ (the case $z=-1$ is similar ) and diverging to $\infty.$ Suppose
for example that $x\in\gamma\cap\left]  1,+\infty\right[  .$ Then
\begin{equation}
\Re\left(  \int_{1}^{x}\sqrt{p_{a}\left(  t\right)  }dt\right)  =\int_{1}%
^{x}\Re\left(  \sqrt{t-a}\right)  \sqrt{t^{2}-1}dt=0. \label{integ}%
\end{equation}
As in the proof of lemma \ref{sets sigma}, if $\Re\left(  \sqrt{t-a}\right)
=0$ for some $t_{0}>1,$ then $\left(  t_{0}-a\right)  <0,$ which cannot hold.
By an adequate choice of the square roots, and continuity of the function
$t\longmapsto\Re\left(  \sqrt{t-a}\right)  ,$ each of the quantities
$\Re\left(  \sqrt{t-a}\right)  $ and $\sqrt{t^{2}-1}$ preserve the same sign
for every $t>1,$ and the integral $\int_{1}^{x}\Re\left(  \sqrt{t-a}\right)
\sqrt{t^{2}-1}dt$ cannot vanish, which contradicts (\ref{eq integ}). If a
critical trajectory emerges from $z=a,$ and intersects $%
\mathbb{R}
\setminus\left\{  \pm1\right\}  $ in two different points, then, obviously
from lemma \ref{teich lemma}, they should be both on a same side with $\pm1.$
Again we get a contradiction by considering the integral in (\ref{integ})
between these two points.

For the last point, let us begin by the case $\Re a\leq0,$ and suppose that
such a critical trajectory $\gamma$ of $\varpi_{a,0}$ does exist. Let be
$bi,b>0$ the first intersection of $\gamma$ with $i%
\mathbb{R}
^{+}.$ Taking the paths of integration in $\left[  0,1\right]  $ and $\left[
0,bi\right]  $ respectively in the first and in the second integral, the
arguments in $\left[  0,2\pi\right[  $ and the square roots in $\left[
0,\pi\right[  ,$ we get%
\begin{align}
0  &  =\Re\left(  \int_{1}^{bi}\sqrt{p_{a}\left(  z\right)  }dz\right)
=-\Re\left(  \int_{0}^{1}\sqrt{p_{a}\left(  z\right)  }dz\right)  +\Re\left(
\int_{0}^{bi}\sqrt{p_{a}\left(  z\right)  }dz\right) \nonumber\\
&  =-\Re\left(  \int_{0}^{1}\sqrt{\left(  1-t^{2}\right)  \left(  a-t\right)
}dt\right)  -b\Im\left(  \int_{0}^{1}\sqrt{\left(  1+b^{2}t^{2}\right)
\left(  a-ibt\right)  }dt\right)  . \label{re de a negative}%
\end{align}
We have for any $t\in\left[  0,1\right]  ,$%
\begin{align*}
\arg\left(  \sqrt{\left(  1-t^{2}\right)  \left(  a-t\right)  }\right)   &
\in\left[  \frac{\pi}{4},\frac{\pi}{2}\right]  ,\\
\arg\left(  \sqrt{\left(  1+b^{2}t^{2}\right)  \left(  a-ibt\right)  }\right)
&  \in\left[  \frac{\pi}{4},\frac{3\pi}{4}\right]  ,
\end{align*}
and then
\[
\Re\left(  \int_{0}^{1}\sqrt{\left(  1-t^{2}\right)  \left(  a-t\right)
}dt\right)  >0,\Im\left(  \int_{0}^{1}\sqrt{\left(  1+b^{2}t^{2}\right)
\left(  a-ibt\right)  }dt\right)  >0
\]
which violates (\ref{re de a negative}).

The case $\Re a\in\left]  0,1\right]  $ can be treated by the same idea, with
the choice of the arguments in $\left]  -\pi,\pi\right[  $ and the square
roots in $\left]  -\frac{\pi}{2},\frac{\pi}{2}\right[  .$\bigskip
\end{proof}

In the following proposition we give a detailed description of the critical
graph $\Gamma_{a,0}$ of $\varpi_{a,0}$ by summarizing the preceding results.
The general case of $\theta$ can be treated in the same vein but more
fastidious discussions and calculations.

\begin{proposition}
\label{prop princ}For any $a\in%
\mathbb{C}
^{+},$ the quadratic differential $\varpi_{a,0}$ has five half-plane domains,
and one or two strip domains, as follows :

\begin{enumerate}
\item If $a$ $\in\Omega_{1},$ a strip domain, bordered by two pairs of
critical trajectories emerging from $z=-1$ and $z=a,$ and diverging to
$\infty$ following the critical directions $D_{0,1}$ and $D_{0,3};$ it
contains $\left]  \Delta,-1\right[  ,\Delta<-1.$ A second strip domain,
bordered by two pairs of critical trajectories emerging from $z=-1$ and
$z=+1,$ and diverging to $\infty$ following the critical directions $D_{0,0}$
and $D_{0,3};$ it contains $\left]  -1,1\right[  .$

\item If $a$ $\in\Sigma_{-1,0}^{l},$ a unique strip domain, bordered by the
short trajectory, two pairs of critical trajectories emerging from $z=-1$ and
$z=a,$ and diverging to $\infty$ following the critical directions $D_{0,0}$
and $D_{0,3};$ it contains $\left]  -1,1\right[  .$

\item If $a$ $\in\Omega_{2},$ a strip domain in the upper half plane, bordered
by two pairs of critical trajectories emerging from $z=-1$ and $z=a,$ and
diverging to $\infty$ following the critical directions $D_{0,0}$ and
$D_{0,2}.$ A second strip domain, bordered by two pairs of critical
trajectories emerging from $z=-1$ and $z=+1,$ and diverging to $\infty$
following the critical directions $D_{0,0}$ and $D_{0,3};$ it contains
$\left]  -1,1\right[  .$

\item If $a$ $\in\Sigma_{-1,0}^{r},$ a unique strip domain, bordered by two
pairs of critical trajectories emerging from $z=-1$ and $z=+1,$ and diverging
to $\infty$ following the critical directions $D_{0,0}$ and $D_{0,3};$ it
contains $\left]  -1,1\right[  .$

\item If $a$ $\in\Omega_{3},$ a strip domain, bordered by two pairs of
critical trajectories emerging from $z=-1$ and $z=a,$ and diverging to
$\infty$ following the critical directions $D_{0,1}$ and $D_{0,3};$ it
contains $\left]  -1,\Delta\right[  ,$ $\Delta\in\left]  -1,1\right[  .$ A
second strip domain, bordered by two pairs of critical trajectories emerging
from $z=+1$ and $z=a,$ and diverging to $\infty$ following the critical
directions $D_{0,0}$ and $D_{0,3};$ it contains $\left]  \Delta,1.\right[  $

\item If $a$ $\in\Sigma_{1,0},$ a unique strip domain, bordered by the short
trajectory, two pairs of critical trajectories emerging from $z=-1$ and
$z=+1,$ and diverging to $\infty$ following the critical directions $D_{0,1}$
and $D_{0,3};$ it contains $\left]  -1,1\right[  .$

\item if $a$ $\in\Omega_{4},$ a strip domain, bordered by two pairs of
critical trajectories emerging from $z=-1$ and $z=+1,$ and diverging to
$\infty$ following the critical directions $D_{0,1}$ and $D_{0,3};$ it
contains $\left]  -1,1\right[  .$ A second strip domain; bordered by two pairs
of critical trajectories emerging from $z=+1$ and $z=a,$ and diverging to
$\infty$ following the critical directions $D_{0,1}$ and $D_{0,4};$ it
contains $\left]  1,\Delta\right[  ,$ $\Delta>1.$
\end{enumerate}
\end{proposition}

\begin{proof}
Let be $\gamma_{1}\left(  t\right)  ,\gamma_{2}\left(  t\right)  ,$ and
$\gamma_{3}\left(  t\right)  ,$ $t\geq0,$ the three critical trajectories of
$\varpi_{a,0}$ emerging from $z=-1,$ with $\gamma_{1}\left(  0\right)
=\gamma_{2}\left(  0\right)  =\gamma_{3}\left(  0\right)  =-1.$ Using
notations of theorem.\ref{main thm} proof, we claim with
lemma.\ref{intersects} that there exists $\alpha_{k}\left(  t\right)
\in\left]  0,\frac{2\pi}{3}\right[  $ such that :
\[
\arg\left(  \gamma_{k}\left(  t\right)  +1\right)  =\alpha_{k}\left(
t\right)  +\frac{2\left(  k-1\right)  \pi}{3},t\longrightarrow0^{+};k=1,2,3.
\]
Since $\arg\left(  a+1\right)  \in\left]  0,\pi\right[  ,$ we get%
\[
\alpha_{1}\left(  0^{+}\right)  =\frac{-\arg\left(  a+1\right)  +\pi}{3}%
\in\left]  0,\frac{\pi}{3}\right[  ,
\]
which shows with lemma \ref{intersects} that $\gamma_{1}$ and $\gamma_{2}$
belong the upper half plane, while $\gamma_{3}$ belongs the lower half plane,
and then, it diverges to $\infty.$ The same reasoning shows that from $z=1$
emerges a trajectory, say $\beta_{1},$ staying in the upper half plane; while
the two remaining trajectories, taken in the counter clockwise orientation,
say $\beta_{2},\beta_{3},$ they belong to the lower half plane and diverge to
$\infty.$ We deduce that $\gamma_{3}$ diverges to $\infty$ following the
critical direction $D_{0,2}$ or $D_{0,3}.$

Let us assume first that $\gamma_{2}$ diverges to $\infty.$ In the large,
$\gamma_{2}$ forms with $\gamma_{3}$ an angle equal to $\frac{2k\pi}{5}$ for a
certain $k=0,...,4.$ Consider the curve formed by $\gamma_{2}$ and $\gamma
_{3}$ as the border of domain $\Omega$ encircling $\left]  -\infty,-1\right[
.$ Applying lemma \ref{teich lemma}, equality (\ref{teich equality}) becomes
\begin{equation}
1+k=2+\sum m_{i}. \label{conseq teich}%
\end{equation}
Lemma \ref{intersects} asserts that the only possible critical point in
$\Omega$ is $z=a,$ in this case, $k=2;$ otherwise, $k=1.$

\begin{enumerate}
\item If $k=2$ and $\gamma_{3}$ diverges to $\infty$ following $D_{0,2},$ then
$\gamma_{1}$ and $\gamma_{2}$ should diverge to $\infty$ following the same
critical direction $D_{0,0},$ which is impossible by lemma \ref{teich lemma}.
Then, $\gamma_{3}$ and $\gamma_{2}$ diverge to $\infty$ respectively following
$D_{0,3}$ and $D_{0,1},\gamma_{1}$ and $\beta_{1}$ diverge towards to $\infty$
following $D_{0,0},$ while $\beta_{2}$ and $\beta_{3}$ diverge to $\infty$
respectively following $D_{0,3}$ and $D_{0,4}.$ Critical trajectories that
emerge from $z=a,$ say, $\delta_{1},\delta_{2},$ and $\delta_{3}$ diverge to
$\infty$ respectively following $D_{0,1},D_{0,2},$ and $D_{0,3}.$ The last one
intersects the real line in some $\Delta<-1.$ In this case, $\varpi_{a,0}$ has
two strip domains, the first one contains $\left]  \Delta,-1\right[  ,$ and it
is bordered by $\gamma_{2},\gamma_{3},\delta_{1}\ $and $\delta_{2}.$ The
second one contains $\left]  -1,1\right[  $, and it is bordered by $\gamma
_{1},\gamma_{3},\beta_{1}$ and $\beta_{2};$ see Fig.\ref{FIG3}.

\item If $k=1,$ then $\Omega$ is a half plane domain of $\varpi_{a,0}$
containing $\left]  -\infty,-1\right[  .$ Suppose that $\gamma_{3}$ diverges
to $\infty$ following the critical direction $D_{0,2}.$ Lemma
\ref{teich lemma} imposes on $\gamma_{2}$ to follow the critical direction
$D_{0,1},$ which is impossible; indeed :

If $\gamma_{1}$ is not a short trajectory, then, again we get equality
(\ref{conseq teich}) for the domain bordered by $\gamma_{1}$ and $\gamma_{2}$
in the upper half plane. It implies that $\gamma_{1}$ diverges to $\infty$
following $D_{0,0}.$ By the other hand, $\beta_{1}$ cannot be a short
trajectory meeting $z=a,$ because if the two critical trajectories emerging
from $z=a$ diverge to $\infty$ in the upper half plane, then their only
possible critical direction is $D_{0,0},$ which violates lemma
\ref{teich lemma}; moreover, lemma \ref{intersects} guarantees that none of
these critical trajectories can intersect the real line. Thus, $\beta_{1}%
$diverges to $\infty$ following $D_{0,0}.$ Again, lemma \ref{teich lemma}
applied on the domain containing $z=0,$ bordered by $\beta_{1},\beta
_{1},\gamma_{1}$ and $\gamma_{3},$ with lemma \ref{intersects} require that
$z=a$ should belong to the domain containing $\left]  1,+\infty\right[  $,
bordered by $\beta_{1}$ and $\beta_{2};$ and we get the same contradiction.

If $\gamma_{1}$ is a short trajectory meeting $z=a,$ then, considering the two
remaining critical trajectories emerging from $z=a,$ we get the same previous contradiction.

We just proved that $\gamma_{3}$ diverges to $\infty$ following $D_{0,3},$
$\gamma_{2}$ diverges to $\infty$ following $D_{0,2},$ and then, $\beta_{2}$
and $\beta_{3}$ diverge to $\infty$ following respectively $D_{0,3}$ and
$D_{0,4}.$ There are three possibilities for $\gamma_{1},$ either it is a
short trajectory, or it diverges to $\infty$ following $D_{0,0}$ or $D_{0,1}.$

\begin{enumerate}
\item If $\gamma_{1}$ is a short trajectory (meeting $z=a$), then the two
remaining critical trajectories emerging from $z=a$ diverge to $\infty$
following respectively $D_{0,0}$ and $D_{0,1};$ $\beta_{1}$ diverge to
$\infty$ following $D_{0,0}.$ In this case, $\left]  -1,1\right[  $ is
included in a strip domain of $\varpi_{a,0}$ defined by $\beta_{1},\beta
_{2},\gamma_{3},\gamma_{1},$ and the critical trajectory emerging from $z=a$
and diverging to $\infty$ following $D_{0,0};$ see Fig.\ref{FIG4}.

\item If $\gamma_{1}$ diverges to $\infty$ following $D_{0,0},$ then from
lemma \ref{teich lemma}, $z=a$ belongs to the domain in the upper half plane
bordered by $\gamma_{1}$ and $\gamma_{2},$ critical trajectories emerging from
$z=a$ diverge to $\infty$ respectively following $D_{0,0},D_{0,1},$ and
$D_{0,2};$ $\beta_{1}$ diverges to $\infty$ following $D_{0,0}.$ In this case,
$\varpi_{a,0}$ has two strip domains, the first, bordered by $\gamma
_{1},\gamma_{2},$ and critical trajectories emerging from $z=a$ that follow
$D_{0,0}$ and $D_{0,2}.$ The second, bordered by $\beta_{1},\beta_{2}%
,\gamma_{3},$ and $\gamma_{1};$ it contains $\left]  -1,1\right[  ;$ see
Fig.\ref{FIG5}.

\item If $\gamma_{1}$ diverges to $\infty$ following $D_{0,1},$ then, either
$\beta_{1}$ is a short trajectory, or it diverges to $\infty$ following
$D_{0,0}$ or $D_{0,1};$

\begin{enumerate}
\item if $\beta_{1}$ is a short trajectory, the two remaining critical
trajectories emerging from $z=a$ diverge to $\infty$ respectively following
$D_{0,0},D_{0,1}.$ In this case, $\left]  -1,1\right[  $ is included in a
strip domain of $\varpi_{a,0}$ defined by $\gamma_{2},\gamma_{3},\beta
_{1},\beta_{2},$ and the critical trajectory emerging from $z=a$ and diverging
to $\infty$ following $D_{0,1};$ see Fig.\ref{FIG6}.

\item if $\beta_{1}$ diverges to $\infty$ following $D_{0,0},$ then two
critical trajectories emerging from $z=a,$ say $\delta_{1},\delta_{2},$
diverge to $\infty$ following respectively $D_{0,0},D_{0,1};$ the remaining
one, say $\delta_{3},$ diverges to $\infty$ following $D_{0,3},$ it intersects
the real line in some $\Delta\in$ $\left]  -1,1\right[  .$ In this case, the
sets $\left]  -1,\Delta\right[  $ and $\left]  \Delta,1\right[  $ are included
in two strip domains, respectively bordered by $\gamma_{1},\gamma_{3}%
,\delta_{2},\delta_{3},$ and $\beta_{1},\beta_{2},\delta_{1},\delta_{2};$ see
Fig.\ref{FIG7}.

\item if $\beta_{1}$ diverges to $\infty$ following $D_{0,1},$ then from lemma
\ref{teich lemma}, $z=a$ belongs the upper half plane bordered by $\beta_{1}$
and $\beta_{3},$ two critical trajectories emerge from $z=a,$ say $\delta
_{1},\delta_{2},$ diverge to $\infty$ following respectively $D_{0,0}$ and
$D_{0,1};$ the third one diverges to $\infty$ following $D_{0,4};$ it
intersects the real line in some $\Delta>1;$ it forms with $\beta_{1}%
,\beta_{3},$ and $\delta_{2},$ a strip domain containing $\left]
1,\Delta\right[  .$ The set $\left]  -1,1\right[  $ is included in a second
strip domain bordered by $\gamma_{1},\gamma_{3},\beta_{1},$ and $\beta_{2};$
see Fig.\ref{FIG8}.
\end{enumerate}
\end{enumerate}
\end{enumerate}

If $\gamma_{2}$ is a short trajectory meeting $z=a,$ and $\gamma_{3}$ diverges
to $\infty$ following $D_{0,2},$ then the two critical trajectories emerging
from $z=a$ diverge to $\infty$ following $D_{0,0}$ and $D_{0,1};$ one of them
diverges together with $\gamma_{1}$ following $D_{0,0,}$ which violates lemma
\ref{teich lemma}. Thus, $\gamma_{3}$ diverges to $\infty$ following
$D_{0,3},$ two critical trajectories emerging from $z=a$ diverge to $\infty$
following $D_{0,1}$ and $D_{0,2},$ $\gamma_{1}$ and $\beta_{1}$ diverge to
$\infty$ following $D_{0,0}.$ The set $\left]  -1,1\right[  $ is included in
the unique strip domain of $\varpi_{a,0}$ defined by $\beta_{1},\beta
_{2},\gamma_{3},$and $\gamma_{1}$; see Fig.\ref{FIG9}.

To finish the proof, the previous reasoning shows that the knowledge of the
behavior of critical trajectories emerging from $z=a\in\Omega_{i}\cap%
\mathbb{C}
^{+},i=1,...,8,$ suffices to determine the whole structure of the critical
graph $\Gamma_{a,0}$ of $\varpi_{a,0}$ in $\Omega_{i}\cap%
\mathbb{C}
^{+}.$ Therefore, it is enough to see the structure of $\Gamma_{a,0}$ when
$\Im a\longrightarrow0^{+}.$ For example, it is straightforward to check that
for $a<-1,$ critical trajectories emerging from $z=a$ diverge to $\infty$
following the critical directions $D_{0,1},D_{0,2}$ and $D_{0,3};$ by
continuity of the structure of trajectories, we deduce the structure of
critical graphs $\Gamma_{a,0}$ in the correspondent sub-set $\Omega_{i}\cap%
\mathbb{C}
^{+};$ see Fig.\ref{FIG9bis}.
\end{proof}

\begin{figure}[tbh]
\begin{minipage}[b]{0.3\linewidth}
		\centering\includegraphics[scale=0.3]{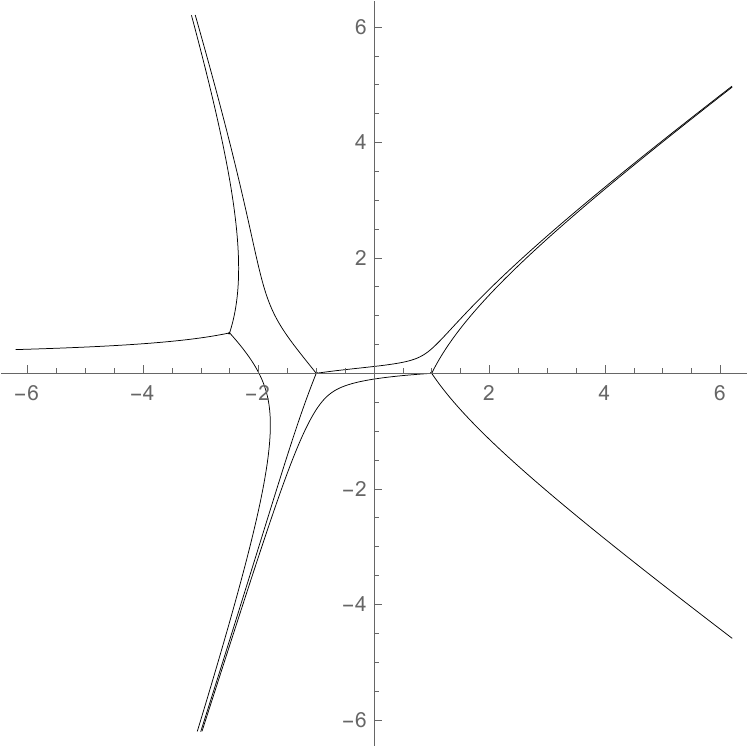}
		\caption{$\Gamma_{-2.5+0.7i,0}$}\label{FIG3}
\end{minipage}\hfill\begin{minipage}[b]{0.3\linewidth}
	\centering\includegraphics[scale=0.30]{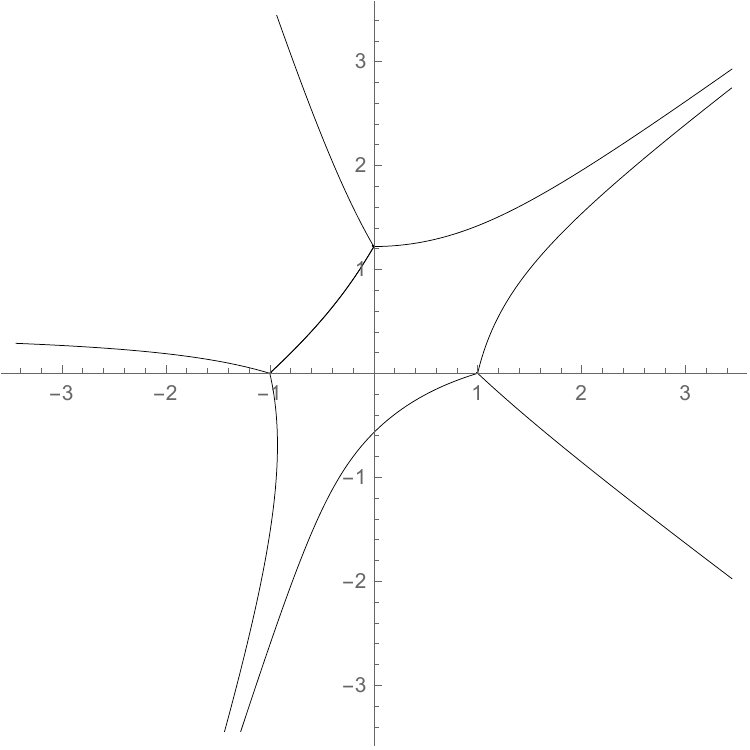}
	\caption{$\Gamma_{-1.22i,0}$}\label{FIG4}
\end{minipage}\hfill\begin{minipage}[b]{0.3\linewidth}
			\centering\includegraphics[scale=0.30]{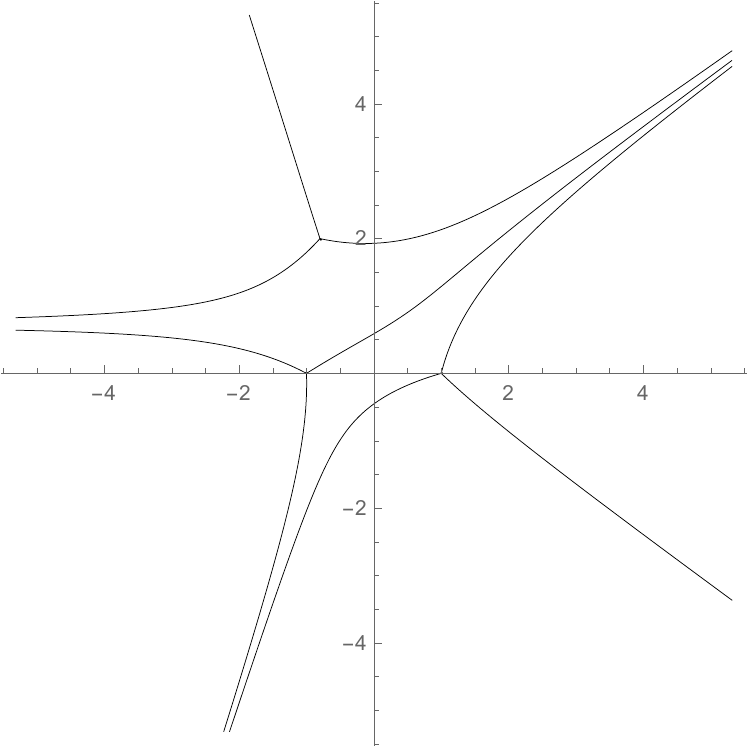}
			\caption{$\Gamma_{-0.8+2i,0}$}\label{FIG5}
\end{minipage}\hfill\begin{minipage}[b]{0.3\linewidth}
	\centering\includegraphics[scale=0.30]{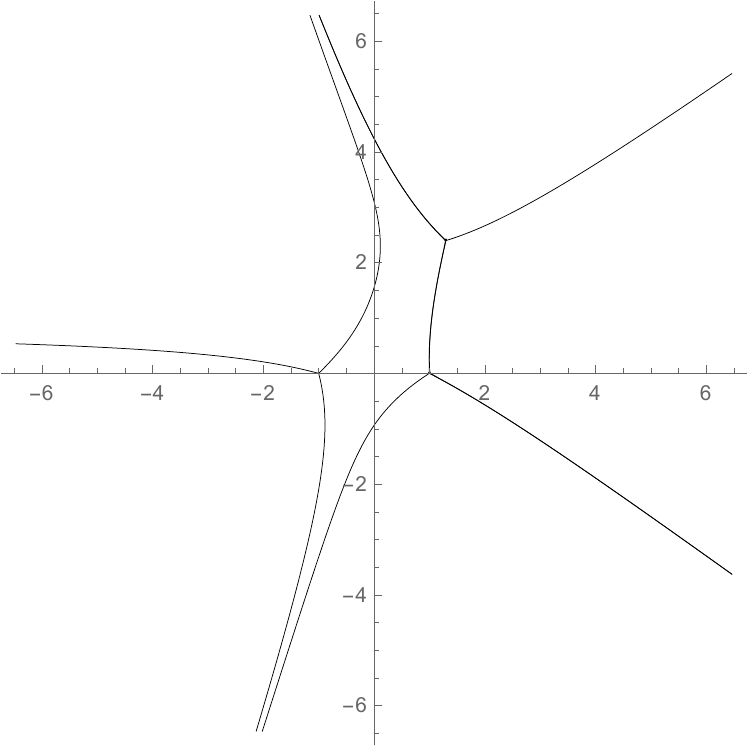}
	\caption{$\Gamma_{1.3+2.4i,0}$}\label{FIG6}
\end{minipage}\hfill\begin{minipage}[b]{0.3\linewidth}
	\centering\includegraphics[scale=0.30]{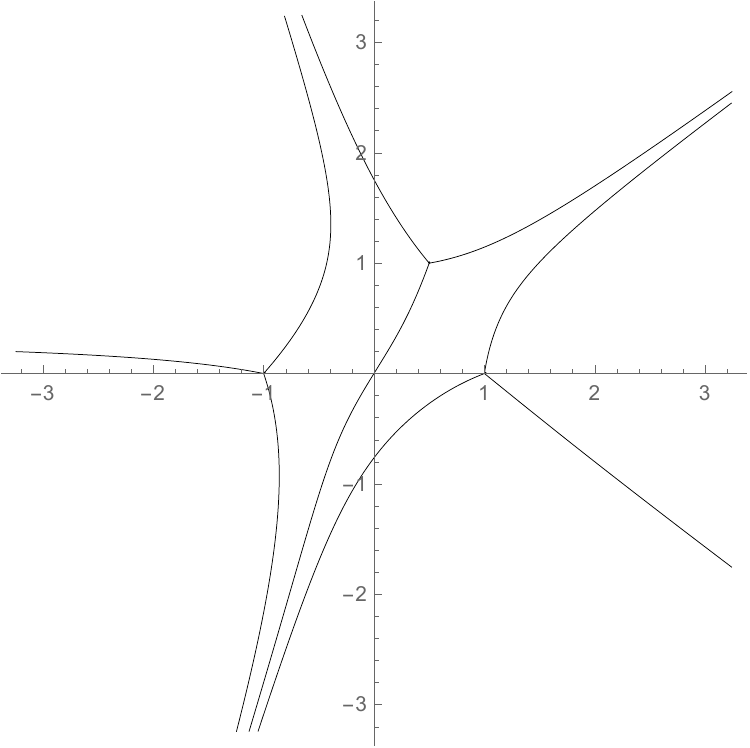}
	\caption{$\Gamma_{0.5+i,0}$}\label{FIG7}
\end{minipage}\hfill\begin{minipage}[b]{0.3\linewidth}
	\centering\includegraphics[scale=0.30]{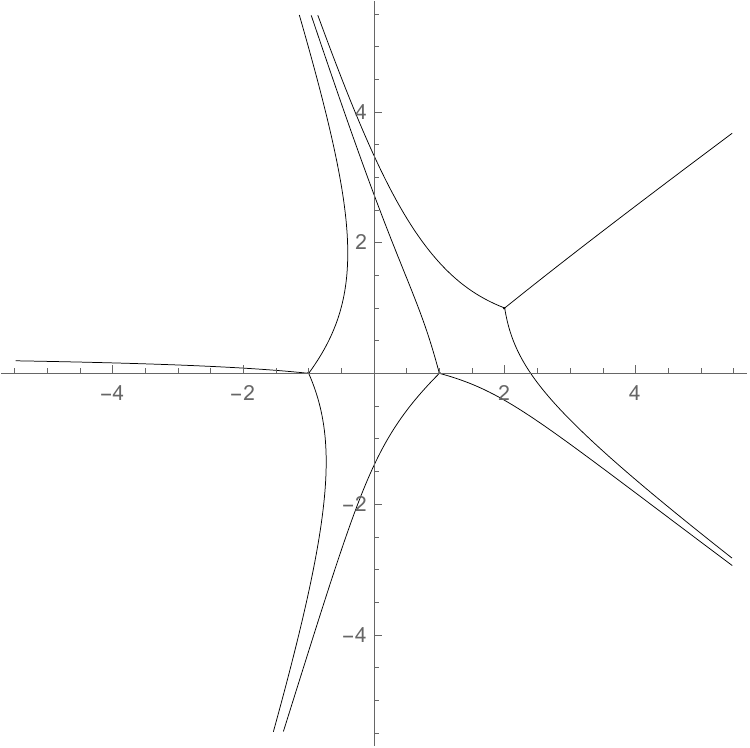}
	\caption{$\Gamma_{2+i,0}$}\label{FIG8}
\end{minipage}\hfill\begin{minipage}[b]{0.3\linewidth}
	\centering\includegraphics[scale=0.30]{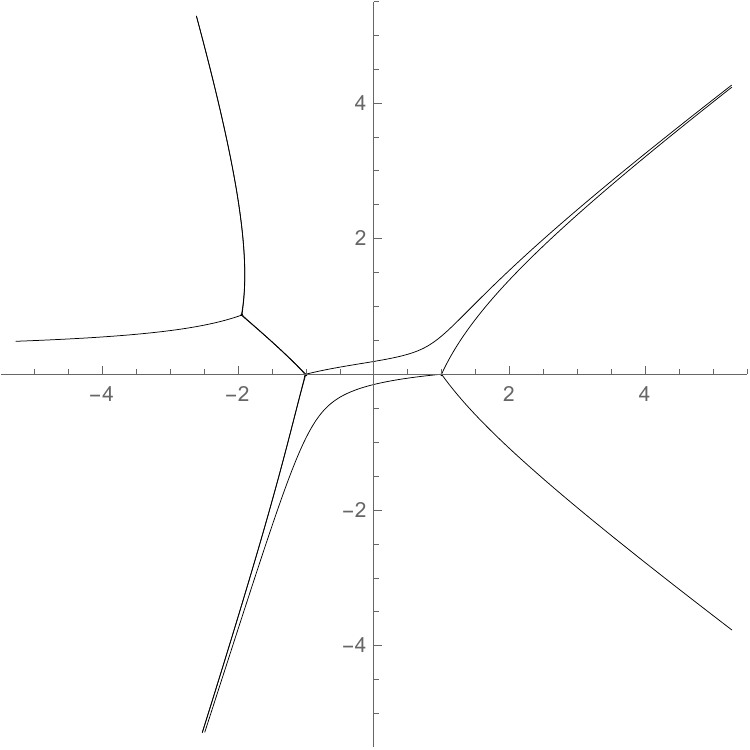}
	\caption{$\Gamma_{-1.95+0.87i,0}$}\label{FIG9}
\end{minipage}\hfill\begin{minipage}[b]{0.3\linewidth}
\centering\includegraphics[scale=0.30]{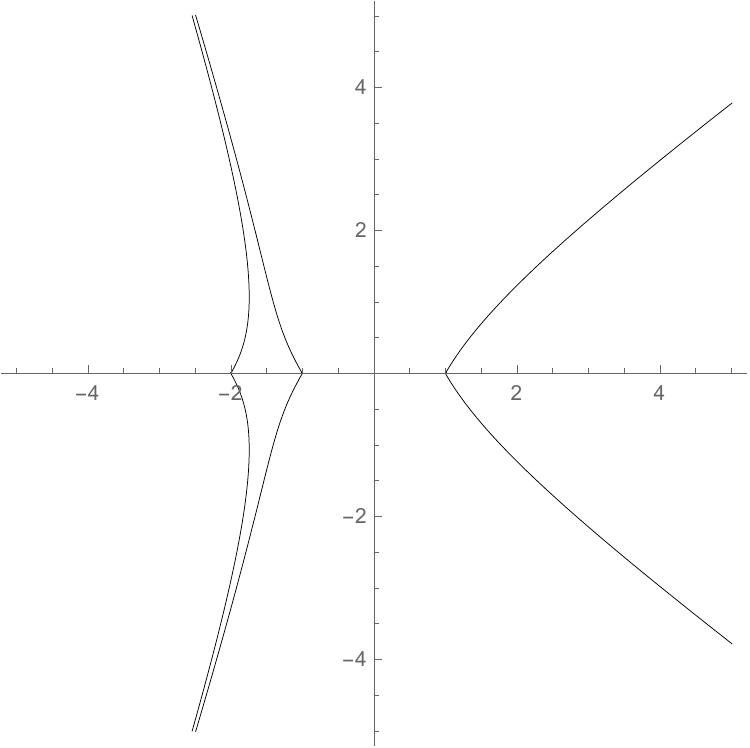}
\caption{$\Gamma_{-2,0}$}\label{FIG9bis}
\end{minipage}\hfill\end{figure}

\bigskip

\subsection{\bigskip The case $\theta=\pi/4$}

In the following proposition, we describe the critical graph of $\varpi
_{a,\pi/4}$ when it has at least one short trajectory. The symmetry of
$\Gamma_{a,\pi/4}$ and $\Gamma_{-\overline{a},\pi/4}$ with respect to the
imaginary axis let us restrict the investigation for $a\in\Sigma_{1,\pi/4}\cup
S_{\pi/4}.$ The proof is a same as in proposition.\ref{prop princ}.

\begin{proposition}
\label{main propo pi/4}Let $a\in\Sigma_{1,\pi/4}\cup S_{\pi/4}.$ Then, besides
five half-plane domains, the configuration domain of \bigskip$\varpi_{a,\pi
/4}$ is as follow :

\begin{enumerate}
\item For $a\in\Sigma_{1,\pi/4}^{r},$ exactly one strip domain; it includes
$\left]  -1,1\right[  $ and it is bounded by, the short trajectory of
$\varpi_{a,\pi/4},$ two critical trajectories diverging to $\infty$ following
the direction $D_{\pi/4,1},$ and two critical trajectories diverging to
$\infty$ following the direction $D_{\pi/4,3};$ see Fig.\ref{FIG 30};

\item \label{case2}for $a\in\Sigma_{1,\pi/4}^{l}\cap\left\{  z:\Re
z>0\right\}  ,$ exactly one strip domain, it includes $\left]  -1,\delta
\right[  $ for some $\delta\in\left]  -1,1\right]  ,$and it is bounded by two
critical trajectories diverging to $\infty$ following the direction
$D_{\pi/4,1},$ and two critical trajectories diverging to $\infty$ following
the direction $D_{\pi/4,3};$ see Fig.\ref{FIG 40};

\item $\varpi_{a,\pi/4}$ has a tree with summit $a_{\pi/4},$ and no strip
domain; see Fig.\ref{FIG 50};

\item for $a\in\Sigma_{1,\pi/4}^{l}\cap\left\{  z:\Re z<0\right\}  ,$ exactly
one strip domain that includes $\left]  \delta,1\right[  $ for some $\delta
\in\left[  -1,1\right[  ,$ and it is bounded by, the short trajectory of
$\varpi_{a,\pi/4},$ two critical trajectories diverging to $\infty$ following
the direction $D_{\pi/4,2},$ and two critical trajectories diverging to
$\infty$ following the direction $D_{\pi/4,4};$ see Fig.\ref{FIG 60}%
,Fig.\ref{FIG 70};

\item for $a=xi,x>\left\vert a_{\pi/4}\right\vert ,$ $\varpi_{a,\pi/4}$ has
exactly one short trajectory connecting $\pm1;$ it is included in $\left\{
z:-1<\Re z<1,0<\Im z<\left\vert a_{\pi/4}\right\vert \right\}  ;$ exactly one
strip domain that is bounded by the short trajectory of $\varpi_{a,\pi/4},$
two critical trajectories diverging to $\infty$ following the direction
$D_{\pi/4,0},$ and two critical trajectories diverging to $\infty$ following
the direction $D_{\pi/4,2};$ see Fig.\ref{FIG 80}.
\end{enumerate}

The short trajectory of $\varpi_{a,\pi/4}$ is a path between two half plane
domains in all cases excepted the case.\ref{case2}, where it joins the strip
domain and a half plane domain. \begin{figure}[tbh]
\begin{minipage}[b]{0.3\linewidth}
		\centering\includegraphics[scale=0.3]{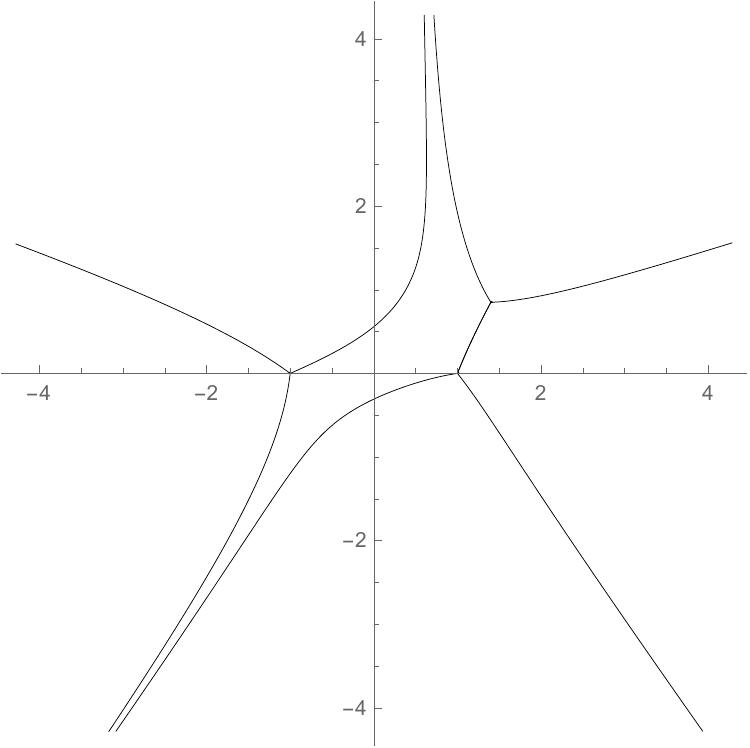}
		\caption{$\Gamma_{1.4+0.85i,\frac{\pi}{4}}$}\label{FIG 30}
	\end{minipage}\hfill\begin{minipage}[b]{0.3\linewidth}
		\centering\includegraphics[scale=0.30]{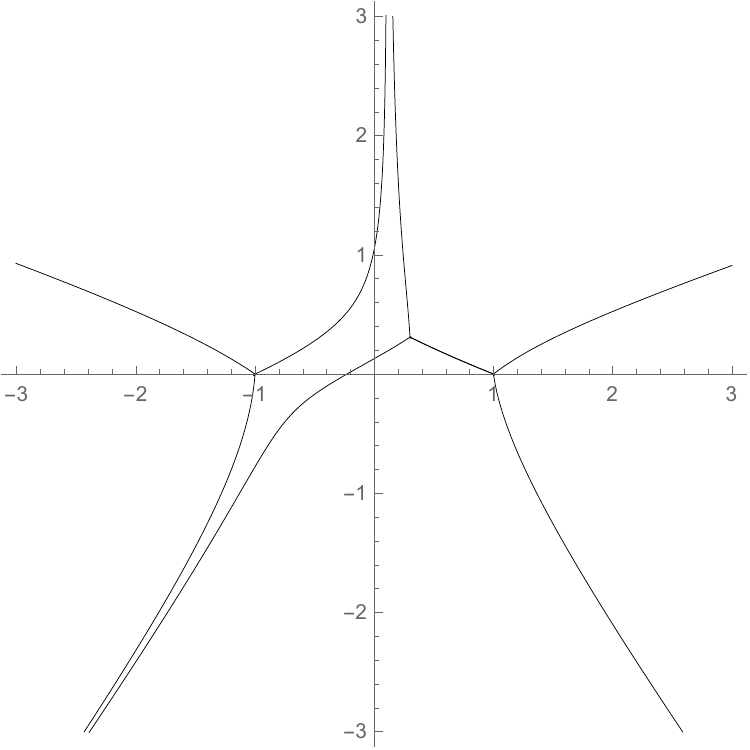}
		\caption{$\Gamma_{0.3 + 0.31i,\pi/4}$}\label{FIG 40}
	\end{minipage}\hfill\begin{minipage}[b]{0.3\linewidth}
		\centering\includegraphics[scale=0.30]{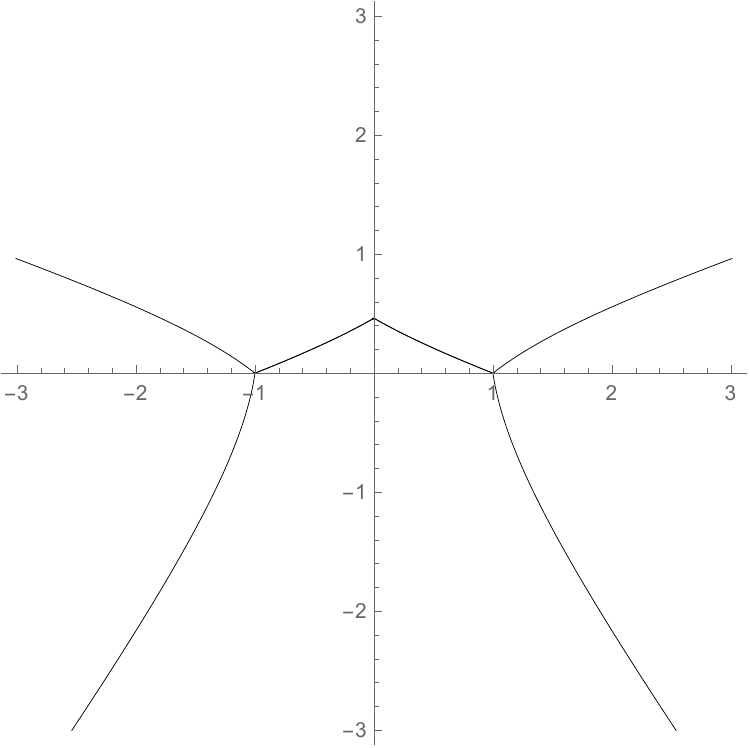}
		\caption{$\Gamma_{0.462i,\pi/4}$}\label{FIG 50}
	\end{minipage}\hfill\begin{minipage}[b]{0.3\linewidth}
		\centering\includegraphics[scale=0.30]{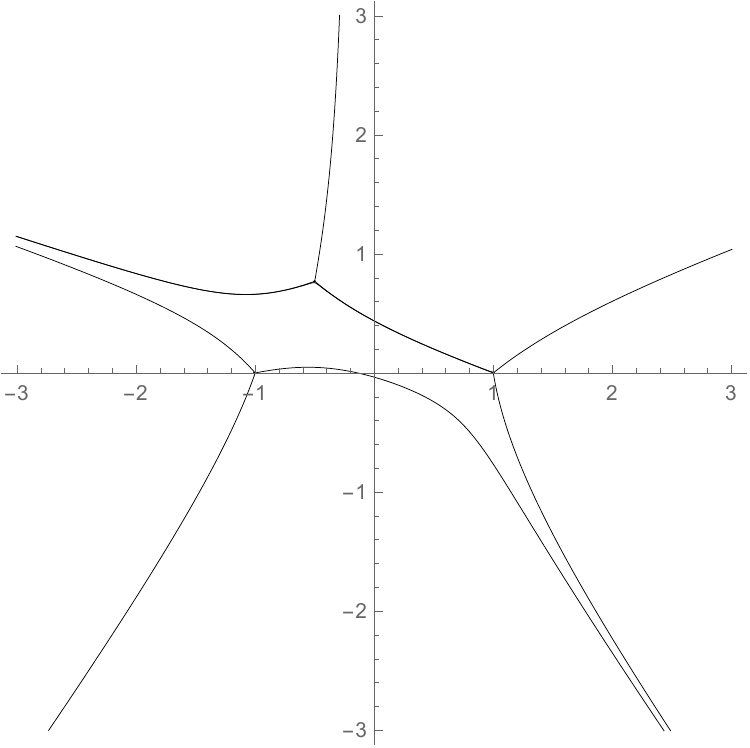}
		\caption{$\Gamma_{-0.5+0.77i,\pi/4}$}\label{FIG 60}
	\end{minipage}\hfill\begin{minipage}[b]{0.3\linewidth}
\centering\includegraphics[scale=0.30]{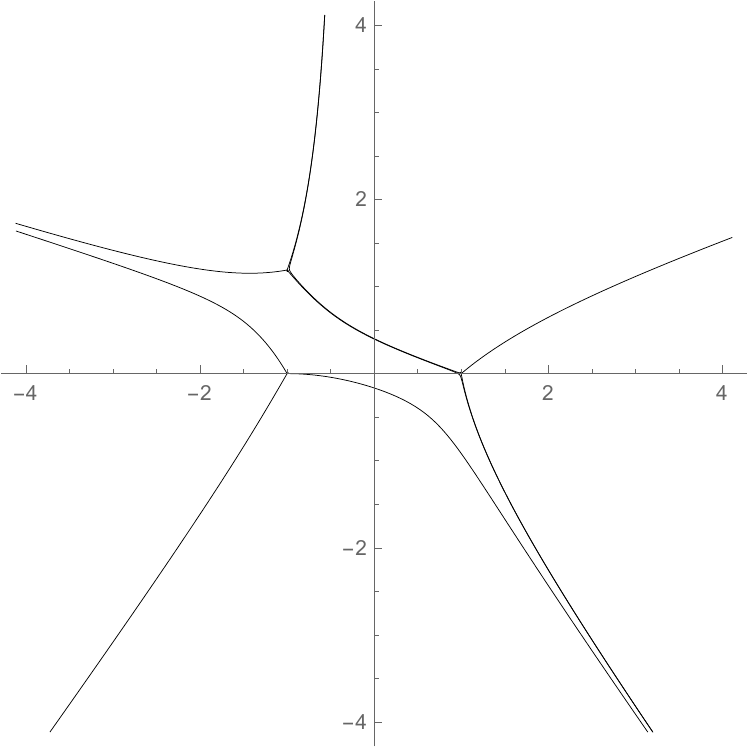}
		\caption{$\Gamma_{-1+1.19i,\pi/4}$}\label{FIG 70}
	\end{minipage}\hfill\begin{minipage}[b]{0.3\linewidth}
	\centering\includegraphics[scale=0.30]{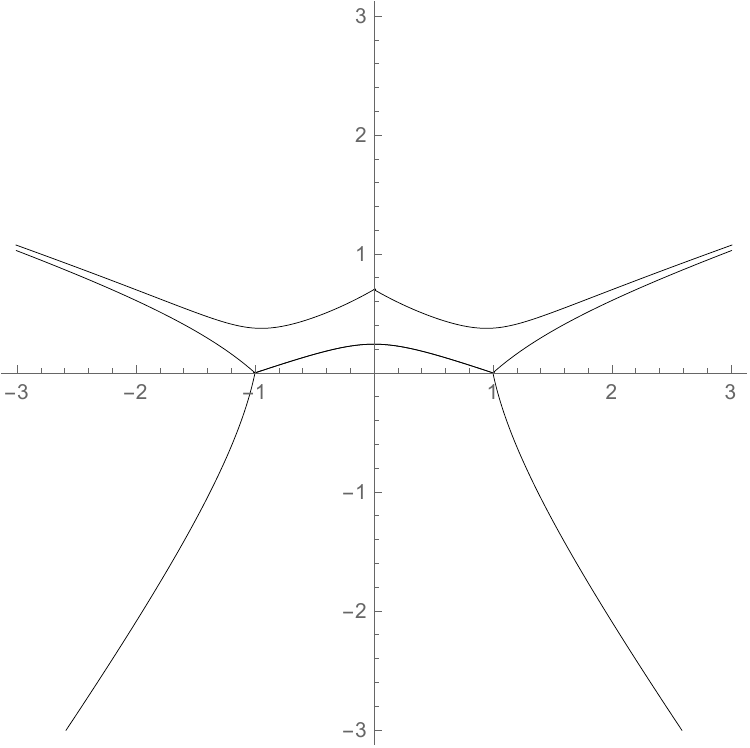}
	\caption{$\Gamma_{0.7i,\pi/4}$}\label{FIG 80}
\end{minipage}\hfill\end{figure}
\end{proposition}

\begin{remark}
The case of real cubic polynomial quadratic differential $\varpi$ with three
real zeros is trivial; when the polynomial has two conjugate complex zeros :%
\[
\varpi_{x_{0},b}=\left(  z-x_{0}\right)  \left(  z-b\right)  \left(
z-\overline{b}\right)  dz^{2},x_{0}\in%
\mathbb{R}
,b\in%
\mathbb{C}
\setminus%
\mathbb{R}
,
\]
by the linear change of variable\bigskip\ : $z=i\Im\left(  b\right)  y+\Re b,$
the quadratic differential $\lambda\varpi,\lambda\in%
\mathbb{C}
^{\ast},$ has the same critical graph then
\[
-e^{2i\left(  \alpha+\pi/4\right)  }\left(  y^{2}-1\right)  \left(
y-i\frac{\Re b-x_{0}}{\Im b}\right)  dy^{2},\alpha\in\left[  0,\pi\right[  .
\]
In particular, $\varpi_{x_{0},b}$ has a tree if, and only if,
\[
i\frac{\Re b-x_{0}}{\Im b}=t_{\pi/4}.
\]

\end{remark}

\begin{summary}
\label{summary}

\begin{enumerate}
\item By (\ref{sym1},\ref{sym2}), we may focus on the case $\theta\in$
$[0,\frac{\pi}{4}]$ without loss of the generality. Theorem.\ref{main thm}
implies that $\Xi_{\theta}$ defined by (\ref{our set imp}) is exactly the
union of level sets $S_{1,\theta},$ $S_{-1,\theta}$ and $S_{\theta}.$ Let
\begin{align*}
\mathcal{D}  &  =\left\{  \left(  \theta,a\right)  \in\left[  0,\pi\right[
\times%
\mathbb{C}
\setminus\left\{  \pm1\right\}  \right\}  ;\\
\left[  \Gamma_{a,\theta}\right]   &  =\left\{  \Gamma_{b,\theta};\text{
}a\mathcal{R}b\text{ }\right\}  ;\\
\mathcal{SG}  &  \mathcal{=}\left\{  \left[  \Gamma_{a,\theta}\right]  ;\text{
}\theta\in\lbrack0,\frac{\pi}{4}]\right\}  ;
\end{align*}
where $\mathcal{R}$ is the equivalence relation defined on $%
\mathbb{C}
\setminus\left\{  \pm1\right\}  $ by :
\[
a\mathcal{R}b\text{ if and only if, critical graphs }\Gamma_{a,\theta}\text{
and }\Gamma_{b,\theta}\text{ have the same structure.}%
\]
We deduce that the map
\[%
\begin{array}
[c]{cc}%
\mathcal{J}:\mathcal{D\longrightarrow SG} & \text{ }\left(  \theta,a\right)
\longmapsto\left[  \Gamma_{a,\theta}\right]
\end{array}
\]
is surjective.

\item From (\ref{maosero-cubic}), let
\[%
\begin{array}
[c]{c}%
V(x)=4x^{3}-2ax-28b\\
\mathcal{D=}\left\{  (a,b)\in%
\mathbb{C}
^{2};\text{ }2a^{3}-1323b^{2}\neq0\right\}  .
\end{array}
\]
The $\mathcal{D\diagup SG}$ correspondence gives a full classification to
Stokes graph of quadratic differential $\varpi(x)=V(x)dx^{2}$ as $(a,b)$
varies in $\mathcal{D},$ which answers (\ref{first question}). As a
consequence, the phase transition in the $\mathcal{D\diagup SG}$
correspondence shows that for a given $a\in%
\mathbb{C}
,$ there is at most two values of $b$ such that $(a,b)\in\mathcal{D},$ and the
Stokes graph of $\varpi$ is a Boutroux curve; this answers
(\ref{second question}).

\item In \cite[probelm 3]{shapiro ev} introduced very flat differential as
polynomial quadratic differential $\exp(it)P(z)dz^{2}$ with simple zeros and a
maximal number of strip domains. Theorem.\ref{main thm} proves that in the
case $\deg P=3,$ there exists $t\in\lbrack0,2\pi\lbrack$ such that
$\exp(it)P(z)dz^{2}$ is very flat.
\end{enumerate}
\end{summary}


\texttt{University of Gabes, Preparatory Institute of Engineering }

\texttt{of Gabes, Avenue Omar Ibn El Khattab 6029. Gabes. Tunisia.}

\texttt{braekgliia@gmail.com}

\texttt{chouikhi.mondher@gmail.com}

\texttt{faouzithabet@yahoo.fr}\bigskip

\bigskip

\end{document}